\documentclass[11pt]{amsart}

\usepackage{amsmath,amssymb,amsfonts}
\usepackage{enumitem,hyperref}
\usepackage{tropical}
\usepackage{xspace}
\usepackage{tikz}
\usepackage[disable,colorinlistoftodos,bordercolor=orange,backgroundcolor=orange!20,linecolor=orange,textsize=tiny]{todonotes}
\usepackage{fullpage}

\usetikzlibrary{arrows,patterns}
\usetikzlibrary{shapes.geometric}
\usetikzlibrary{intersections}
\tikzstyle{dashdotdotdotted}=[dash pattern=on 3pt off 2pt on \the\pgflinewidth off 2pt on \the\pgflinewidth off 2pt on \the\pgflinewidth off 2pt]
\tikzstyle{dashdotdotdotdotted}=[dash pattern=on 3pt off 2pt on \the\pgflinewidth off 2pt on \the\pgflinewidth off 2pt on \the\pgflinewidth off 2pt on \the\pgflinewidth off 2pt]

\providecommand{\arxiv}[2][]{\href{http://www.arxiv.org/abs/#2}{arXiv:#2}}

\newtheorem{theorem}{Theorem}
\newtheorem{proposition}[theorem]{Proposition}
\newtheorem{lemma}[theorem]{Lemma}
\newtheorem{corollary}[theorem]{Corollary}

\newtheorem{remark}[theorem]{Remark}

\newtheorem{example}[theorem]{Example}

\newcommand{\tproj}{\mathbb{T}\mathbb{P}}
\newcommand{\tprojp}{\tilde{\mathbb{T}}\mathbb{P}}
\newcommand{\R}{\mathbb{R}}

\newcommand{\T}{\mathbb{T}}
\newcommand{\Tp}{\tilde{\mathbb{T}}}
\newcommand{\K}{\mathbb{K}}
\newcommand{\Kp}{\tilde{\mathbb{K}}}

\newcommand{\NN}{\mathcal{N}}

\newcommand{\PPp}{\tilde{\mathcal{P}}}

\newcommand{\mpinverse}[1]{-{#1}}

\newcommand{\mydef}{:=}

\newcommand{\leqlex}{\leq_{\textrm{lex}}}
\newcommand{\hahn}[1]{\R[\![t^{#1}]\!]}

\DeclareMathOperator{\tconv}{\mathsf{tconv}}
\DeclareMathOperator{\conv}{\mathsf{conv}}
\DeclareMathOperator{\cone}{\mathsf{cone}}

\DeclareMathOperator{\tper}{\mathsf{tper}}

\renewcommand{\mptimes}{\odot}
\newcommand{\utimes}{}
\DeclareMathOperator{\val}{\mathsf{val}}
\renewcommand{\SS}{\mathscr{S}}

\DeclareMathAlphabet\mathbfcal{OMS}{cmsy}{b}{n}

\newcommand{\liftP}{\mathbfcal{P}}
\newcommand{\liftPp}{\mathbfcal{\tilde{P}}}

\newcommand{\liftH}{\mathbfcal{H}}
\newcommand{\liftHp}{\mathbfcal{\tilde{H}}}
\newcommand{\liftvp}{\tilde{\mathbold{v}}}
\newcommand{\liftxp}{\tilde{\mathbold{x}}}
\newcommand{\liftdp}{\tilde{\mathbold{d}}}
\newcommand{\liftfp}{\tilde{\mathbold{f}}}

\newcommand{\liftx}{\mathbold{x}}
\newcommand{\lifty}{\mathbold{y}}
\newcommand{\liftv}{\mathbold{v}}

\newcommand{\liftf}{\mathbold{f}}
\newcommand{\vp}{\tilde{v}}
\newcommand{\ap}{\tilde{a}}
\newcommand{\xp}{\tilde{x}}

\newcommand{\Ap}{\tilde{A}}

\newcommand{\lift}[1]{\mathbold{#1}}

\usetikzlibrary{calc}

\title{Tropicalization of facets of polytopes}
\thanks{The authors were partially supported by the PGMO program of EDF and
  Fondation Math\'ematique Jacques Hadamard. An important part of this work was done during a visit of Ricardo D.~Katz to the Ecole Polytechnique, which was funded by Digit{\'e}o. Ricardo D.~Katz was also partially supported by CONICET Grant PIP 112-201101-01026.}
  
\author{Xavier {A}llamigeon}
\address{INRIA and CMAP, \'Ecole Polytechnique, CNRS, 91128 Palaiseau Cedex, France}
\email{xavier.allamigeon@inria.fr}
\author{Ricardo D.~Katz}
\address{CONICET-CIFASIS, Bv. 27 de Febrero 210 bis, 2000 Rosario, Argentina}
\email{katz@cifasis-conicet.gov.ar}

\catcode`\@=11
\@namedef{subjclassname@2010}{%
  \textup{2010} Mathematics Subject Classification}
\catcode`\@=12

\begin{document}
 
\begin{abstract}
It is known that any tropical polytope is the image under the valuation map of ordinary polytopes over the Puiseux series field. The latter polytopes are called lifts of the tropical polytope. We prove that any pure tropical polytope is the intersection of the tropical half-spaces given by the images under the valuation map of the facet-defining half-spaces of a certain lift. We construct this lift explicitly, taking into account geometric properties of the given polytope. Moreover, when the generators of the tropical polytope are in general position, we prove that the above property is satisfied for any lift. This solves a conjecture of Develin and Yu. 
\end{abstract}

\maketitle

\section{Introduction} 

The {\em max-plus semiring} $\maxplus$, sometimes also referred to as the \emph{tropical semiring}, consists of the set $\R\cup \{-\infty\}$ equipped with $x\mpplus y\mydef \max \{x,y\}$ as addition and $x\mptimes y \mydef x+y$ as multiplication. The set $\maxplus^n$ is a semimodule over the max-plus semiring when equipped with the component-wise tropical addition $(u,v) \mapsto u \mpplus v\mydef (u_1 \mpplus v_1,\ldots ,u_n \mpplus v_n)$ and the tropical scalar multiplication $(\lambda,u) \mapsto \lambda \mptimes u \mydef (\lambda \mptimes u_1,\ldots ,\lambda \mptimes u_n)$. For our purposes, it turns out to be more convenient to restrict ourselves to points with finite coordinates, and work in the {\em tropical projective space} $\tproj^{n-1}\mydef \R^n/(1,\ldots ,1) \R$ modding out by tropical scalar multiplication. 

The tropical analogues of ordinary polytopes in the Euclidean space are defined as the tropical convex hulls of finite sets of points in the tropical projective space. More precisely, the {\em tropical polytope} generated by the points 
$v^1,\dots , v^p$ of $\tproj^{n-1}$ is defined as 
\begin{equation}\label{EqDefTropPoly}
\PP =\tconv \left( \{ v^1,\dots , v^p \} \right)\mydef \left\{(\lambda_1 \mptimes v^1) \oplus \cdots \oplus (\lambda_p \mptimes v^p) \in \tproj^{n-1}\mid 
\lambda_1, \ldots , \lambda_p \in \R \right\} \; . 
\end{equation}
A tropical polytope is said to be {\em pure} when it coincides with the closure of its interior (with respect to the topology of $\tproj^{n-1}$ induced by the usual topology of $\R^n$). We refer to Figure~\ref{fig:running_example} below for an example (for visualization purposes, we represent the point $(x_1, \ldots, x_n)$ of $\tproj^{n-1}$ by the point $(x_2- x_1, \ldots ,x_n - x_1)$ of $\R^{n-1}$).

The interest in tropical polytopes (also known as semimodules or max-plus cones), and more generally in tropically convex sets, comes from different areas, which include optimization~\cite{zimmerman77}, idempotent functional analysis~\cite{litvinov00} and control theory~\cite{ccggq99}.
They are also interesting combinatorial objects, with connections to, for example, subdivisions of a product of simplices and polyhedral complexes~\cite{DS}, monomial ideals~\cite{blockyu06,DevelinYu}, and oriented matroids~\cite{ArdilaDevelin}. Some algorithmic aspects of tropical polytopes have also been established, see for instance~\cite{joswig-2008,AllamigeonGaubertGoubaultDCG2013}. 

As in classical convexity, tropical polytopes are precisely the bounded intersections of finitely many tropical half-spaces, \ie\ they also admit finite external representations. A {\em tropical half-space} is defined as a set of the form 
\begin{equation}\label{EqDefTropHS}
\left\{x\in \tproj^{n-1}\mid \mpplus_{i\in I} c_i \mptimes x_i \geq \mpplus_{j\in J} c_j \mptimes x_j \right\} \; ,  
\end{equation}  
where $I,J$ is a non-trivial partition of $\{1,\dots ,n\}$ and $c_k\in \R$ for $k\in \{1,\dots ,n\}$. Tropical polytopes also arise as the images under the valuation map of ordinary polytopes over non-archimedean ordered fields, such as the fields of real Puiseux or Hahn series. In this context, a \emph{lift} of a tropical polytope $\PP$ is any polytope $\liftP$ over the non-archimedean ordered field whose image under the valuation map is $\PP$ (see Section~\ref{SubSectionLifts} for details). 

A recurrent problem in tropical geometry is to study whether we can obtain a representation of a given ``tropical object'' by tropicalizing the representation of a certain lift, \ie~by considering the image under the valuation map of this representation. This has been initially studied in the case of tropical varieties, which are defined as the image under the valuation map of 
algebraic varieties over a non-archimedean and algebraically closed field. 
In more details, it has been shown that any ideal of polynomials over such a field has a \emph{tropical basis}, \ie~it has a finite generating set $\{f_1,\dots ,f_m\}$ such that the tropical variety associated with the ideal is given by the intersection of the tropical hypersurfaces associated with the polynomials in $\{f_1,\dots ,f_m\}$. We refer to~\cite{MaclaganSturmfels2015} and the references therein for further details.

In the case of polytopes, it is natural to study the images under the valuation map of the external representations of lifts  provided by their facet-defining half-spaces. An additional motivation for this is the fact that the notion of faces, in particular of facets, is not yet well understood in the tropical setting, see~\cite{DevelinYu}. 
In this regard, Develin and Yu conjectured in~\cite[Conjecture~2.10]{DevelinYu} the following result, the proof of which is the main contribution of this paper.
\begin{theorem}\label{ConjDevelinYu}
For any pure tropical polytope $\PP$ there exists a lift $\liftP$ such that the images under the valuation map of the facet-defining half-spaces of $\liftP$ are tropical half-spaces whose intersection is $\PP$. If, in addition, the generators of $\PP$ are in general position, then any lift works.
\end{theorem}
Using the notation of Theorem~\ref{ConjDevelinYu}, this establishes that, when the tropical polytope $\PP$ is pure, the following diagram commutes:
\begin{center}
\begin{tikzpicture}[>=stealth']
\node[text width=4cm,text centered] (n1) at (0,0) {facet-defining half-spaces};
\node[text width=4cm,text centered] (n2) at (7,0) {convex polytope $\liftP$};
\node[text width=4cm,text centered] (n3) at (0,-2) {tropical half-spaces};
\node[text width=4cm,text centered] (n4) at (7,-2) {tropical polytope $\PP$};

\draw[->] (n1) -- node[above,font=\small] {intersection} (n2);
\draw[->] (n3) -- node[below,font=\small] {intersection} (n4);
\draw[->] (n1) -- node[left,font=\small] {valuation} (n3);
\draw[->] (n2) -- node[right,font=\small] {valuation} (n4);
\end{tikzpicture}
\end{center}
A few such commutation results on tropical polytopes have been established so far. Before stating their conjecture, Develin and Yu show in~\cite[Proposition~2.9]{DevelinYu} that if the intersection of the tropical half-spaces obtained from the facet-defining half-spaces is pure, then the diagram above commutes. Independently, Allamigeon et al.\ have shown in~\cite{AllamigeonBenchimolGaubertJoswigSIDMA2015} that this also holds when these tropical half-spaces are in a general position. In fact, the latter result can be recovered from the one of Develin and Yu, since the general position assumption implies the purity condition of~\cite[Proposition~2.9]{DevelinYu}, see~\cite[Theorem~45]{JoswigLoho} for a proof.

This paper is organized as follows. Section~\ref{SectionPrelim} is devoted to recalling results and concepts that are needed in this paper. In Section~\ref{SectionPropPure} we establish some combinatorial properties of pure tropical polytopes, which are essential in our proof Theorem~\ref{ConjDevelinYu}. Notably, we identify a particular external representation for pure polytopes which is used in Section~\ref{SectionGeneralP} to prove the second part of  Theorem~\ref{ConjDevelinYu}. The part of Theorem~\ref{ConjDevelinYu} concerning generators not necessarily in general position, which is dealt with in Section~\ref{SectionArbitraryP}, is more difficult to prove because in that case not every lift works. To prove it we use the result when the generators are in general position together with a symbolic perturbation technique, where the perturbation is based on some geometric properties of the tropical polytope. Finally, it is worth mentioning that our proof of Theorem~\ref{ConjDevelinYu} is constructive.

\begin{figure}
\begin{tikzpicture}[convex/.style={draw=black,ultra thick,fill=lightgray,fill opacity=0.7},point/.style={blue!50},>=triangle 45,
hsborder/.style={green!60!black,dashdotted,thick},
hs/.style={draw=none,pattern=north east lines,pattern color=green!60!black,fill opacity=0.5},
apex/.style={green!60!black},
shsborder/.style={orange,dashdotted,thick},
shs/.style={draw=none,pattern=north west lines,pattern color=orange,fill opacity=0.5},
]
\draw[gray!30,very thin] (-0.5,-0.5) grid (8.5,6.5);
\draw[gray!50,->] (-0.5,0) -- (8.5,0) node[color=gray!50,below] {$x_2-x_1$};
\draw[gray!50,->] (0,-0.5) -- (0,6.5) node[color=gray!50,above] {$x_3-x_1$};

\coordinate (v1) at (1,3);
\coordinate (v12) at (3,3);
\coordinate (v2) at (3,1);
\coordinate (v23) at (5,1);
\coordinate (v3) at (6,2);
\coordinate (v34) at (6,5);
\coordinate (v4) at (2,5);
\coordinate (v41) at (1,4);

\begin{scope}
\clip (-0.5,-0.5) rectangle (8.5,6.5);
\begin{scope}[shift={(-1,-1)}]
\filldraw[hs] (-10,3) -- (3,3) -- (3,-10) -- (3.3,-10) -- (3.3,3.3) -- (-10,3.3) -- cycle;
\draw[hsborder] (-10,3) -- (3,3) -- (3,-10);
\filldraw[apex] (3,3) circle (2.5pt) node[below left] {$a$};
\end{scope}

\draw[shsborder] (6,-10) -- (6,2) -- (10,6);
\filldraw[shs] (6,-10) -- (6,2) -- (10,6) -- (10,5.576) -- (6.3,1.876) -- (6.3,-10) -- cycle;
\end{scope}

\filldraw[convex] (v1) -- (v12) -- (v2) -- (v23) -- (v3) -- (v34) -- (v4) -- (v41) -- cycle;
\filldraw[point] (v1) circle (3pt) node[below left=1pt] {$v^1$};
\filldraw[point] (v2) circle (3pt) node[below left=1pt] {$v^2$};
\filldraw[point] (v3) circle (3pt) node[below right=3pt] {$v^3$};
\filldraw[point] (v4) circle (3pt) node[above left=1pt] {$v^4$};

\end{tikzpicture}
\caption{A pure tropical polytope (in grey), and the half-space $\HH(a,I)$ with $a = (0,2,2)$ and $I = \{2,3\}$ (in green). The extreme points $v^1, \dots, v^4$ of the polytope are depicted in blue. The point $v^3 = (0,6,2)$ is extreme of type $2$ because it is the only point of the polytope contained in the sector $\SS(v^3,2)$ (in orange).}\label{fig:running_example}
\end{figure}
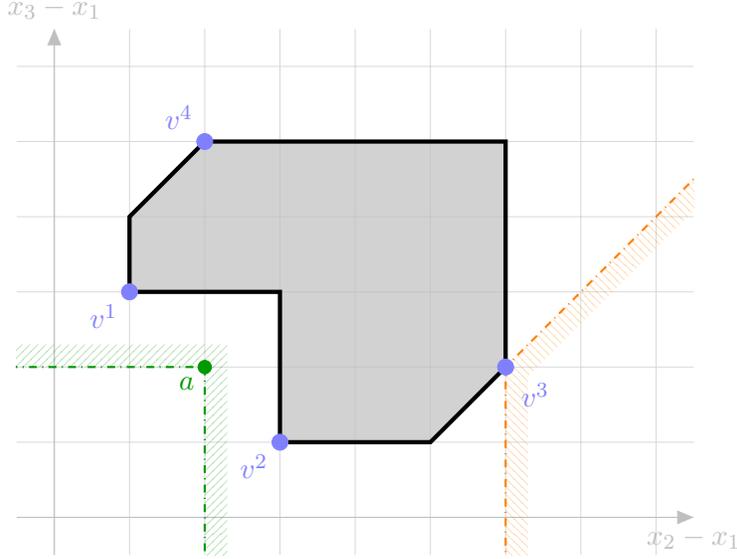

\section{Preliminaries}\label{SectionPrelim}

Given a natural number $n$, in what follows we denote the set $\{1,\ldots ,n\}$ by $\oneto{n}$.

The {\em tropical permanent} of a square matrix $U = (u_{ij})\in \maxplus^{n \times n}$ is defined by 
\begin{equation}\label{eq:Deftperm}
\tper (U) \mydef 
\mpsum_{\sigma \in \Sigma(n)} u_{1\sigma(1)} \mptimes \cdots \mptimes u_{n\sigma(n)} = \max_{\sigma \in \Sigma (n)} (u_{1\sigma(1)} + \dots + u_{n\sigma(n)}) \; , 
\end{equation}
where $\Sigma(n)$ is the set of permutations of $\oneto{n}$. The matrix $U$ is said to be {\em tropically singular} the maximum in~\eqref{eq:Deftperm} is attained at least twice. 
We shall say that the points $v^1,\dots , v^p$ of $\tproj^{n-1}$ are in {\em general position} when every square submatrix of the $n \times p$ matrix whose columns are $v^1,\dots , v^p$ is tropically non-singular.

Consider the tropical half-space given in~\eqref{EqDefTropHS}. The point $a= ( \mpinverse{c}_1,\ldots , \mpinverse{c}_n)\in \tproj^{n-1}$ will be called the {\em apex} of the half-space. This half-space will be then denoted by $\HH(a,I)$, \ie
\[
\HH(a,I)\mydef \left\{ x\in \tproj^{n-1}\mid \mpplus_{i\in I} (\mpinverse{a}_i) \mptimes x_i \geq \mpplus_{j\in \oneto{n}\setminus I} (\mpinverse{a}_j) \mptimes x_j \right\} \; .
\]
See Figure~\ref{fig:running_example} for an illustration. A {\em signed (tropical) hyperplane} is a set of the form 
\[
\left\{ x\in \tproj^{n-1} \mid \mpplus_{i\in I} (-a_i) \mptimes x_i =  \mpplus_{j \in J} (-a_j) \mptimes x_j \right\}\; , 
\]
where $I,J$ is a non-trivial partition of $\oneto{n}$ and $a_k\in \R$ for $k\in \oneto{n}$. It corresponds to the boundary of the tropical half-space $\HH(a,I)$. A {\em tropical hyperplane} is defined as the set of points $x\in \tproj^{n-1}$ where the maximum in a tropical linear form $\oplus_{i\in \oneto{n}} (-a_i) \mptimes x_i =\max_{i\in \oneto{n}} (-a_i) + x_i$ is attained at least twice. In both cases, the point $a =(a_1,\ldots ,a_n)\in \tproj^{n-1}$ is referred to as the {\em apex} of the hyperplane. These objects are illustrated in Figure~\ref{fig:hyperplane}. 

Observe that the complement of the tropical hyperplane with apex $a$ can be partitioned in $n$ connected regions. The closure of each of these regions is a set of the form
\[
\SS(a,i)\mydef \left\{x \in \tproj^{n-1} \mid (\mpinverse{a}_i)\mptimes  x_i \geq  \mpplus_{j\in \oneto{n}\setminus \{ i \} } (\mpinverse{a}_j)\mptimes  x_j \right\} \; , 
\]   
where $i \in \oneto{n}$. This is a special tropical half-space usually known as (closed) \emph{sector}, 
see the right hand-side of Figure~\ref{fig:hyperplane}. The tropical half-space $\HH(a,I)$ is the union of the sectors $\SS(a,i)$ for $i \in I$.
Another property of sectors is the following equivalence, which we will use several times in the rest of the paper:
\begin{equation}\label{PropertySectors}
b\in \SS(a,i) \iff \SS(b,i) \subset \SS(a,i) \; .
\end{equation}

\begin{figure}
\begin{center}
\begin{tikzpicture}[convex/.style={draw=lightgray,fill=lightgray,fill opacity=0.7},convexborder/.style={ultra thick},point/.style={blue!50},>=triangle 45,scale=0.7]
\begin{scope}
\draw[gray!30,very thin] (-0.5,-0.5) grid (9.5,6.5);
\draw[gray!50,->] (-0.5,0) -- (9.5,0) node[color=gray!50,below] {$x_2-x_1$};
\draw[gray!50,->] (0,-0.5) -- (0,6.5) node[color=gray!50,above] {$x_3-x_1$};

\node[coordinate] (a) at (5,3) {};
\node[coordinate] (v3) at (8.5,6.5) {};
\node[coordinate] (v2) at (5,-0.5) {};	
\draw[convexborder] (v2) -- (a) -- (v3);
\filldraw[point] (a) circle (3.5pt) node[above left=1.5pt] {$a$};
\end{scope}
\begin{scope}[xshift=11.9cm]
\draw[gray!30,very thin] (-0.5,-0.5) grid (9.5,6.5);
\draw[gray!50,->] (-0.5,0) -- (9.5,0) node[color=gray!50,below] {$x_2-x_1$};
\draw[gray!50,->] (0,-0.5) -- (0,6.5) node[color=gray!50,above] {$x_3-x_1$};

\node[coordinate] (a) at (5,3) {};
\node[coordinate] (v1) at (-0.5,3) {};
\node[coordinate] (v2) at (5,-0.5) {};
\node[coordinate] (v3) at (8.5,6.5) {};
\draw[convexborder] (a) -- (v1) (a) -- (v2) (a) -- (v3);
\filldraw[point] (a) circle (3.5pt) node[above left=1.5pt] {$a$};
\node at (2,1.3) {$\SS(a,1)$};
\node at (7.5,2.5) {$\SS(a,2)$};
\node at (3.5,5.5) {$\SS(a,3)$};
\end{scope}
\end{tikzpicture}
\end{center}
\caption{Left: the signed hyperplane $\left\{ x\in \tproj^2 \mid x_1 \mpplus (-3) \mptimes x_3 =  (-5) \mptimes x_2 \right\}$, with apex $a = (0,5,3)$. Right: the tropical hyperplane with apex $a$ and the associated sectors.}
\label{fig:hyperplane}
\end{figure}

A set $\SS \subset \tproj^{n-1}$ is said to be \emph{tropically convex} if the point $(\lambda \mptimes x) \mpplus (\mu \mptimes y)$ belongs to $\SS$ for all $x, y \in \SS$ and $\lambda, \mu \in \R$. Tropical polytopes, half-spaces and (signed) hyperplanes are all tropically convex sets.

Henceforth, $\PP\subset \tproj^{n-1}$ denotes the tropical polytope~\eqref{EqDefTropPoly} generated by the points $v^1,\dots , v^p$ (which we assume are all distinct). The {\em type} of $a\in \tproj^{n-1}$ with respect to the points $v^1,\dots , v^p$ is defined as the $n$-tuple $S(a)=(S_1(a),\ldots ,S_n(a))$ where, for each $i\in \oneto{n}$, the set $S_i(a)$ consists of the super-indexes $r\in \oneto{p}$ such that $v^r$ is contained in the sector $\SS(a,i)$:
\begin{align*} 
S_i(a) & \mydef 
\bigl\{ r\in \oneto{p} \mid v^r \in \SS(a,i) \bigr\} = \bigl\{ r\in \oneto{p} \mid (\mpinverse{a}_i) \mptimes v^r_i \geq (\mpinverse{a}_j) \mptimes v^r_j \; \text{ for all}\; j\in \oneto{n} \bigr\} \; .
\end{align*}
Thus, $S(a)$ provides information on the relative position of $a$ with respect to the points $v^1,\dots , v^p$.

Given a $n$-tuple $S=(S_1,\ldots , S_n)$ of subsets of 
$\oneto{p}$, the set $X_S$ composed of all the points whose type contains $S$, \ie  
\[
X_S\mydef \left\{ x\in \tproj^{n-1} \mid S_i\subset S_i(x),\; \forall i\in \oneto{n} \right\} \; ,
\] 
is a closed convex polyhedron, see~\cite[Lemma~10]{DS}. More precisely, we have 
\begin{equation}\label{EqCell}
X_S = \left\{ x\in \tproj^{n-1} \mid (\mpinverse{x}_i) \mptimes x_j \geq (\mpinverse{v}^r_i) \mptimes v^r_j \makebox{ for all } j,i\in \oneto{n} \makebox{ and }  r\in S_i \right\} \; .
\end{equation}   
Recall that the \emph{natural cell decomposition} of $\tproj^{n-1}$ induced 
by the points $v^1,\dots , v^p$ is the polyhedral complex formed by the convex polyhedra $X_S$, where $S$ ranges over all the possible types. For instance, Figure~\ref{fig:cell_decomposition} illustrates the cell decomposition induced by the generators $v^1, \dots, v^4$ of the polytope in Figure~\ref{fig:running_example}. A point $a\in \tproj^{n-1}$ is called a \emph{pseudovertex} of the cell decomposition when $\{ a \}$ is one of its (zero-dimensional) cells. By~\cite[Theorem~15]{DS} it is known that the tropical polytope $\PP$ precisely coincides with the union of the bounded cells. The latter are characterized as the cells $X_S$ such that $S_i \neq \emptyset$ for all $i \in \oneto{n}$.  We refer the reader to~\cite{DS} for a detailed presentation of types and the associated cell decomposition, warning that the results of~\cite{DS} are in the setting of the min-plus semiring, which is however equivalent to the setting considered here. 

\begin{figure}
\begin{tikzpicture}[scale=1.2,convexborder/.style={very thick},point/.style={blue!50},pseudov/.style={red!50},
pseudovIJ/.style={green!50!black},>=triangle 45,convex/.style={draw=none,fill=lightgray,fill opacity=0.2}]
\draw[gray!30,very thin] (-0.5,-0.5) grid (8.5,6.5);
\draw[gray!50,->] (-0.5,0) -- (8.5,0) node[color=gray!50,below] {$x_2-x_1$};
\draw[gray!50,->] (0,-0.5) -- (0,6.5) node[color=gray!50,above] {$x_3-x_1$};

\coordinate (v1) at (1,3);
\coordinate (v12) at (3,3);
\coordinate (v2) at (3,1);
\coordinate (v23) at (5,1);
\coordinate (v3) at (6,2);
\coordinate (v34) at (6,5);
\coordinate (v4) at (2,5);
\coordinate (v41) at (1,4);
\coordinate (u1) at (6,3);
\coordinate (u2) at (3,5);

\begin{scope}
\clip (-0.5,-0.5) rectangle (8.5,6.5);

\begin{scope}[shift={(v1)}]
\draw[convexborder,dashed] (0,0) -- (-20,-20) (0,0) -- (0,20) (0,0) -- (20,0);
\end{scope}
\begin{scope}[shift={(v2)}]
\draw[convexborder,dashdotted] (0,0) -- (-20,-20) (0,0) -- (0,20) (0,0) -- (20,0);
\end{scope}
\begin{scope}[shift={(v3)}]
\draw[convexborder,dashdotdotted] (0,0) -- (-20,-20) (0,0) -- (0,20) (0,0) -- (20,0);
\end{scope}
\begin{scope}[shift={(v4)}]
\draw[convexborder,dashdotdotdotted] (0,0) -- (-20,-20) (0,0) -- (0,20) (0,0) -- (20,0);
\end{scope}
\end{scope}

\fill[convex] (v1) -- (v12) -- (v2) -- (v23) -- (v3) -- (v34) -- (v4) -- (v41) -- cycle;

\node[font=\small,rotate=45] at (0,4.5) {$(\emptyset,\{1,2,3,4\},\emptyset)$};
\node[font=\small,rotate=90] at (1.5,6.2) {$(\{1\},\{2,3,4\},\emptyset)$};
\node[font=\small,rotate=90] at (2.5,6.2) {$(\{1,4\},\{2,3\},\emptyset)$};
\node[font=\small] at (4.5,6) {$(\{1,2,4\},\{3\},\emptyset)$};
\node[font=\small] at (7.5,6) {$(\{1,2,3,4\},\emptyset,\emptyset)$};
\node[font=\small,rotate=45] at (0,2.5) {$(\emptyset,\{1,2,3\},\{4\})$};
\node[font=\small,rotate=45] at (2,4) {$(\{1\},\{2,3\},\{4\})$};
\node[font=\small] at (4.5,4) {$(\{1,2\},\{3\},\{4\})$};
\node[font=\small] at (7.5,4) {$(\{1,2,3\},\emptyset,\{4\})$};
\node[font=\small] at (4.5,2) {$(\{2\},\{3\},\{1,4\})$};
\node[font=\small] at (7.5,2.5) {$(\{2,3\},\emptyset,\{1,4\})$};
\node[font=\small] at (7.5,1.5) {$(\{2\},\emptyset,\{1,3,4\})$};
\node[font=\small] at (1.5,1.5) {$(\emptyset,\{2,3\},\{1,4\})$};
\node[font=\small,rotate=45] at (3,0) {$(\emptyset,\{3\},\{1,2,4\})$};
\node[font=\small] at (6.5,0) {$(\emptyset,\emptyset,\{1,2,3,4\})$};

\filldraw[point] (v1) circle (2.5pt) node[below right=1pt] {$v^1$};
\filldraw[point] (v2) circle (2.5pt) node[above left=1pt] {$v^2$};
\filldraw[point] (v3) circle (2.5pt) node[below right=1pt] {$v^3$};
\filldraw[point] (v4) circle (2.5pt) node[below=3pt] {$v^4$};
\filldraw[pseudovIJ] (v12) circle (2.5pt) node[below left=1pt] {$a^1$};
\filldraw[pseudovIJ] (v23) circle (2.5pt) node[below right=1pt] {$a^2$};
\filldraw[pseudovIJ] (v34) circle (2.5pt) node[above right=1pt] {$a^3$};
\filldraw[pseudovIJ] (v41) circle (2.5pt) node[above left=1pt] {$a^4$};
\filldraw[pseudov] (u1) circle (2.5pt) node[above right=1pt] {$u^1$};
\filldraw[pseudov] (u2) circle (2.5pt) node[above right=1pt] {$u^2$};

\end{tikzpicture}
\caption{The cell decomposition of $\tproj^2$ induced by the points $v^1, \dots, v^4$ (for the sake of readability, only the types of $2$-dimensional cells are provided). This cell decomposition has $10$ pseudovertices which are depicted in blue, green and red. The bounded cells are shaded in gray.}\label{fig:cell_decomposition}
\end{figure}
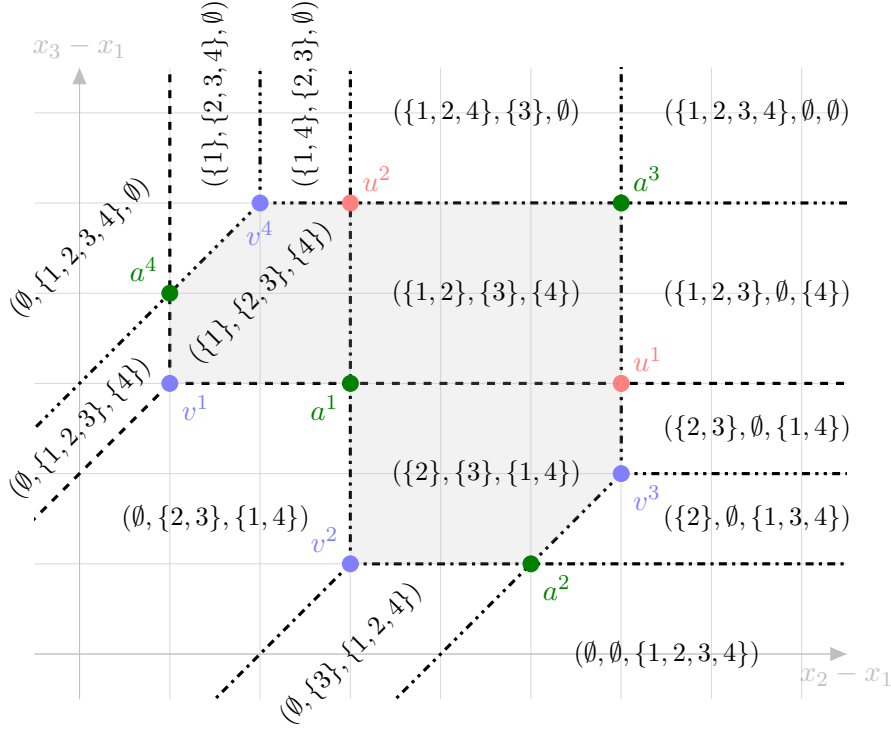

\subsection{Tropical semirings associated with totally ordered abelian groups}\label{SubSectionTG}

In Section~\ref{SectionArbitraryP} we will work with a tropical semiring whose elements represent symbolic perturbations of the elements of the max-plus semiring $\maxplus$. As the max-plus semiring, this tropical semiring arises from an ordered group. In this regard, let us recall that with a totally ordered abelian group $(G, \cdot , \preceq )$ it is possible to associate a tropical semiring which is defined as follows. Its ground set $\T(G)$ consists of the elements of $G$ and an extra element $\bot \not \in G$. The ordering and group law are extended to $\T(G)$ by setting $\bot \preceq \lambda$ and $\bot \cdot  \lambda = \lambda \cdot  \bot = \bot$ for all $\lambda \in G$. The set $\T(G)$ is equipped with the maximum (according to the order $\preceq$) as addition, and the group law $\cdot $ as multiplication. Thus, the max-plus semiring can be obtained in this way by instantiating $(G, \cdot , \preceq )$ as the group $(\R, + ,\leq)$. By abuse of notation, the operations of $\T(G)$ will be still denoted by $\mpplus$ and $\mptimes$. The tropical addition and scalar multiplication can be defined as well over elements of $(\T(G))^n$. As we did before, we introduce the projective space $\tproj^{n-1}(G)$ as the quotient of $G^n$ by the equivalence relation $\sim$ given by $x \sim y$ if $x = \lambda \mptimes y$ for some $\lambda \in G$. 

\subsection{The Hahn series field}\label{SubSectionHahnSeries}

Given a totally ordered abelian group $(G, \cdot , \preceq )$, the field $\hahn{G}$ of {\em real Hahn series with value group $(G, \cdot , \preceq )$} consists of formal power series of the form 
\[
\liftx =\sum_{\alpha \in G} x_\alpha t^\alpha \; ,
\]
where $x_\alpha \in \R$ for $\alpha \in G$ and the set $\{\alpha \in G \mid x_\alpha \neq 0 \}$, called the {\em support} of $\liftx$, is well-ordered.

The {\em valuation} $\val(\liftx)$ of a non-null Hahn series $\liftx$ is defined as $-\alpha_{\min}(\liftx)$, where $\alpha_{\min}(\liftx)$ is the least element of the support of $\liftx$. By convention, we define $\val(\liftx)$ as the null element $\bot$ of $\T(G)$ when $\liftx$ is the identically null series. The valuation map $\liftx\mapsto \val(\liftx)$ extends to vectors component-wise.

We say that a non-null Hahn series $\liftx$ is {\em positive}, and write $\liftx  >0$, when $x_{\alpha_{\min} (\liftx )}>0$. 
We write $\liftx \geq 0$ if $\liftx $ is positive or null, and $\liftx \geq \lifty$ if $\liftx - \lifty\geq 0$. 
Equipped with the sum and product of formal power series, the set $\hahn{G}$ constitutes a totally ordered field. Then, the usual theory of discrete geometry applies in $(\hahn{G})^n$, and in particular we can define polyhedral sets in the usual way. We will denote by $\conv \left( \{\liftv^r\}_{r\in \oneto{p}} \right)$ the convex hull of the points $\liftv^1,\ldots ,\liftv^p\in (\hahn{G})^n$, and by $\cone \left( \{\liftv^r\}_{r\in \oneto{p}} \right)$ their conic hull. 

Observe that, for all $\liftx, \lifty \in \hahn{G}$, we have 
\[
\begin{aligned}
\val(\liftx \utimes \lifty)& =\val(\liftx )\mptimes \val(\lift{y}) \\
\val(\liftx +\lifty ) & \leq \val(\liftx )\oplus \val(\lifty )
\end{aligned}
\]
and that the inequality above turns into an equality in particular when $\liftx$ and $\lifty$ are positive. 

For the sake of simplicity, in what follows we denote by $\K$ the field of real Hahn series with value group $(\R, + ,\leq)$, and by $\K^+$ the set of its positive elements. The field of Puiseux series with real coefficients is obviously a subfield of $\K$.

\subsection{Lifts of tropical polytopes}\label{SubSectionLifts}

As we mentioned in the introduction, tropical polytopes are images under the valuation map of ordinary polytopes over the Hahn series field:

\begin{proposition}[{\cite[Proposition 2.1]{DevelinYu}}]\label{Prop2.1DevelinYu}
Without loss of generality, assume $v^r_1=0$ for all $r\in \oneto{p}$. For each $r\in \oneto{p}$, let $\liftv^r \in (\K^+)^n$ be such that $\val (\liftv^r)=v^r$. Then, $\PP=\val \left( \conv \left( \{\liftv^r\}_{r\in \oneto{p}} \right) \right)$.
\end{proposition}

In Proposition~\ref{Prop2.1DevelinYu} it is assumed that $v^r_1=0$ for all $r\in \oneto{p}$ only for simplicity. Alternatively, we can ignore this condition and define $\liftP=\cone \left( \{\liftv^r\}_{r\in \oneto{p}} \right)\subset \K^n$, where $\liftv^r \in (\K^+)^n$ are such that $\val (\liftv^r)=v^r$ for $r\in \oneto{p}$. Then, the conclusion of Proposition~\ref{Prop2.1DevelinYu} holds for the polyhedral cone $\liftP$, \ie\ we have $\PP=\val (\liftP)$. We will call any such cone $\liftP$ a {\em lift} of $\PP$. In~\cite{DevelinYu} a lift of $\PP$ is any polytope $\conv \left( \{\liftv^r\}_{r\in \oneto{p}} \right)$ over the Puiseux series field defined as in Proposition~\ref{Prop2.1DevelinYu}.  Observe that when the points $\liftv^r$ in  Proposition~\ref{Prop2.1DevelinYu} are taken so that $\liftv^r_1=1$ for all $r\in \oneto{p}$, then $\conv \left( \{\liftv^r\}_{r\in \oneto{p}} \right)$ is nothing but the section of $\liftP$ defined by the hyperplane $\bigl\{ \liftx \in \K^n\mid \liftx_1 =1 \bigr\}$. In consequence, the two definitions are essentially equivalent in what concerns the facial structure of the lift. 

\begin{remark}
Observe that if the points $v^1,\dots , v^p$ are in general position and $\liftv^r \in (\K^+)^n$ are such that $\val (\liftv^r)=v^r$ for $r\in \oneto{p}$, then the points  $\liftv^1,\dots , \liftv^p$ are also in general position, meaning that the determinant of every square submatrix of the matrix whose columns are the points $\liftv^1,\dots , \liftv^p$ is non-null. 
\end{remark} 

The relation between hyperplanes and half-spaces over the Hahn series field and their tropical analogues is given by the following proposition. 

\begin{proposition}[{\cite[Proposition 2.4]{DevelinYu}}]\label{Prop2.4DevelinYu}
Assume $\liftf_i\in \K$ for $i\in \oneto{n}$ are non-null. Then, under the valuation map, the image of the hyperplane $\bigl\{ \liftx \in \K^n\mid \sum_{i\in \oneto{n}} \liftf_i \utimes \liftx_i =0 \bigr\}$ is the tropical hyperplane with apex $a=(\mpinverse{\val (\liftf_1)}, \ldots ,\mpinverse{\val (\liftf_n)})$, the image of $\bigl\{ \liftx \in \K^n\mid \sum_{i\in \oneto{n}} \liftf_i \utimes \liftx_i \geq 0 \bigr\} \cap (\K^+)^n$ is the tropical half-space $\HH(a,I)$, where $I=\{i\in \oneto{n}\mid \liftf_i > 0\}$, and the image of $\bigl\{ \liftx \in \K^n\mid \sum_{i\in \oneto{n}} \liftf_i \utimes \liftx_i =0 \bigr\} \cap (\K^+)^n$ is (the signed tropical hyperplane given by) the boundary of $\HH(a,I)$. 
\end{proposition}

Observe that in order to obtain a tropical half-space (resp.~signed hyperplane), it is necessary to consider only the images of the points of the half-space (resp.~hyperplane) which belong to $(\K^+)^n$. 
For the sake of simplicity, in this paper we refer to the image (under the valuation map) of the intersection of a half-space with $(\K^+)^n$ simply as the \emph{image of the half-space}.  

\section{Properties of pure polytopes}\label{SectionPropPure}

\subsection{Extreme points}

A point $v\in \PP$ is said to be {\em extreme} in the tropical polytope $\PP$ if $v=x\mpplus y$ with $x,y\in \PP$ implies $v=x$ or $v=y$. Thus, the set of extreme points of $\PP$ is a subset of its set of generators $\{ v^r \}_{r\in \oneto{p}}$. 
We recall the following characterization of extreme points. 

\begin{proposition}[{\cite{GK,BSS}}]\label{CharactExtP}
The point $v\in \tproj^{n-1}$ is extreme in the tropical polytope $\PP$ if and only if there exists $i \in \oneto{n}$ such that $v$ is the only point of $\PP$ contained in the sector $\SS(v,i)$, or, equivalently, $v$ is a minimal element of the set $\bigl\{ z\in \maxplus^n \mid  z\in \PP , \, z_i = v_i \bigr\}$ w.r.t.\ the coordinate-wise ordering $\leq$. 
\end{proposition}

In view of Proposition~\ref{CharactExtP}, if $v\in \tproj^{n-1}$ is the only point of $\PP$ contained in the sector $\SS(v,i)$, then $v$ is said to be an \emph{extreme point of $\PP$ of type $i$}. We refer to Figure~\ref{fig:running_example} for an illustration.

\begin{remark}\label{RemarkNeighborhood}  
Note that for any $x \in \PP$ and $i \in \oneto{n}$ there exists an extreme point $v$ of $\PP$ of type $i$ such that $v \in \SS(x,i)$. This follows from the fact that the set $\bigl\{ z\in \maxplus^n \mid z_i =x_i, \, z_j \leq x_j \ \text{for all} \ j\in \oneto{n}\bigr\}\cap \PP$ is compact and non-empty, and so it has a minimal element w.r.t.\ $\leq$. 

In consequence, if $v$ is an extreme point of $\PP$ of type $i$, then there exists a  (open) neighborhood~$\NN$ of $v$ such that $v \in \SS(x,i)$ for all $x \in \PP \cap \NN$. To see this, assume the contrary. Let $(x^k)_k$ be a sequence of points of $\PP$ converging to $v$ and such that $v \not \in \SS(x^k,i)$ for all $k$. Recall that $\PP$ has a finite number of extreme points. Thus, by first part of the remark, there exists an extreme point $u$ of $\PP$ of type $i$ verifying $u \in \SS(x^k,i)$ for infinitely many $k$. As a result, $u \neq v$ and $u \in \SS(v,i)$, contradicting the fact that $v$ is extreme of type $i$. 
\end{remark}

In general, an extreme point can be of several types (for example, the point $v^1$ of the polytope depicted on the left-hand side of Figure~\ref{fig:non-pure} below is extreme of types $1$ and $3$). The next lemma shows that the uniqueness of the extremality type characterizes pure polytopes, a property which will play an important role in our proof of Theorem~\ref{ConjDevelinYu}. 

\begin{lemma}\label{LemmaExtTypePure}
The tropical polytope $\PP$ is pure if, and only if, all its extreme points have a unique extremality type. 
\end{lemma}
 
\begin{proof}
We prove the statement by contraposition. First, let $v$ be an extreme point of $\PP$, and suppose that it is extreme of types $i$ and $j$, with $i \neq j$. By Remark~\ref{RemarkNeighborhood}, we know that there exists a (open) neighborhood $\NN$ of $v$ such that $v \in \SS(x,i)$ and $v \in \SS(x,j)$ for all $x \in \PP \cap \NN$. It follows that $(\mpinverse{v}_i) \mptimes x_i = (\mpinverse{v}_j) \mptimes x_j$ for all $x \in \PP \cap \NN$, and so $\PP \cap \NN$ has an empty interior. Since $\NN$ is open, the interior of $\PP \cap \NN$ is equal to the intersection of $\NN$ with the interior of $\PP$. Hence, no point of $\NN$ can be in the interior of $\PP$. Therefore, $\PP$ is not pure.

Conversely, suppose that $\PP$ is not pure. In this case, the cell decomposition induced by the generators of $\PP$ contains a maximal cell $X_S$ which is not full-dimensional. As $X_S$ is not full-dimensional, by~\cite[Proposition~17]{DS} we know that $S_i \cap S_j \neq \emptyset$ for two distinct indexes $i,j \in \oneto{n}$. By~\eqref{EqCell}, we deduce that the quantity $(\mpinverse{z_i}) \mptimes z_j$ is constant when $z$ ranges over $X_S$. Let $x$ be a point in the relative interior of $X_S$. Since the cell $X_S$ is maximal, there exists a neighborhood $\NN$ of $x$ such that $ (\mpinverse{z_i}) \mptimes z_j= (\mpinverse{x_i}) \mptimes x_j$ for all $z \in \PP \cap \NN$. 

We claim that $\PP \cap \SS(x,i) = \PP \cap \SS(x,j)$. To see this, consider $y \in \PP\cap \SS(x,i)$. By definition, we have $(\mpinverse{x_k}) \mptimes y_k \leq (\mpinverse{x_i}) \mptimes y_i$ for all $k \in \oneto{n}$, and it suffices to show that the equality holds for $k = j$ to ensure $y \in \SS(x,j)$. If the inequality was strict for $k = j$, consider the point $z = x \mpplus (\epsilon \mptimes x_i \mptimes (\mpinverse{y_i})) \mptimes y$, where $x_j\mptimes (\mpinverse{x_i}) \mptimes y_i \mptimes (\mpinverse{y_j})>\epsilon > 0$. Then, $z$ would belong to $\PP$ because $\PP$ is tropically convex. Besides, since $x_k \leq z_k \leq x_k + \epsilon$ for all $k \in [n]$, the point $z$ would also belong to the neighborhood $\NN$ of $x$, provided that $\epsilon$ is small enough. However, $z_i = x_i + \epsilon > x_i$ and $z_j = x_j$, which would contradict the fact that $ (\mpinverse{z_i}) \mptimes z_j= (\mpinverse{x_i}) \mptimes x_j$ for all $z \in \PP \cap \NN$. This proves the inclusion $\PP \cap \SS(x,i) \subset \PP \cap \SS(x,j)$. The other inclusion follows from the symmetry between $i$ and $j$.

Let $\QQ$ be the tropical polytope given by $\PP \cap \SS(x,i) = \PP \cap \SS(x,j)$. Observe that for all $y \in \QQ$, we have $(\mpinverse{y_i}) \mptimes y_j = (\mpinverse{x_i}) \mptimes x_j$. Consequently, any extreme point of $\QQ$ of type $i$ is also extreme of type $j$, and vice versa. Let $v \in \QQ$ be such a point (which exists by Remark~\ref{RemarkNeighborhood}, as $\QQ$ is not empty). To complete the proof, it is enough to show that $v$ is also an extreme point of $\PP$ of types $i$ and $j$. By~\eqref{PropertySectors} we have $\SS(v,i) \subset \SS(x,i)$ and $\SS(v,j) \subset \SS(x,j)$. Therefore, the sets $\PP \cap \SS(v,i) = \QQ \cap \SS(v,i)$ and $\PP \cap \SS(v,j) = \QQ \cap \SS(v,j)$ are both reduced to the singleton $\{v\}$, which shows the desired property. 
\end{proof}

\subsection{External representations and $(I,j)$-pseudovertices}\label{SectionNonredundant}

An external representation of the tropical polytope $\PP$ is any finite set of tropical half-spaces whose intersection is $\PP$. One way of proving Theorem~\ref{ConjDevelinYu} is showing that the set of images under the valuation map of the facet-defining half-spaces of a lift contains a specific external representation of $\PP$. In this section, we build such an external representation of $\PP$, based on the study of the non-redundant representations in~\cite{AllamigeonKatzJCTA2013}.

An external representation of a tropical polytope is said to be {\em non-redundant} when the elimination of any of its half-spaces does not lead to another external representation. We restrict our attention to external representations composed of half-spaces whose apices belong to the polytope. Under this assumption, it was shown in~\cite{AllamigeonKatzJCTA2013} that any tropical polytope has an essentially unique non-redundant external representation. In more details, all non-redundant representations precisely involve the same set of apices, which are called the {\em non-redundant apices} of the polytope. The next theorem provides a characterization of a superset of the set of these apices:

\begin{theorem}[{\cite[Theorem~43]{AllamigeonKatzJCTA2013}}]\label{ThNonRedundantApex}
Let $a$ be a non-redundant apex of $\PP$. Then, there exist a non-empty proper subset $I$ of $\oneto{n}$ and $j\in \oneto{n}\setminus I$ such that the following properties hold:
\begin{enumerate}[label=(\roman*),ref=\roman*]
\item\label{ThNonRedundantApexC1} $\bigcup_{i\in I} S_i(a)=\oneto{p}$; 
\item\label{ThNonRedundantApexC2} for each $k\in \oneto{n} \setminus I$ there exists $i\in I$ such that $S_i(a)\cap S_k(a) \neq \emptyset$; 
\item\label{ThNonRedundantApexC3} $S_i(a)\cap S_j(a) \not \subseteq \bigcup_{h\in I\setminus \{i \} } S_h(a)$ for all $i\in I$.   
\end{enumerate}
A point $a\in \tproj^{n-1}$ satisfying the three properties above is a pseudovertex of the cell decomposition induced by the generators of $\PP$ which will be called  a \emph{$(I,j)$-pseudovertex} of $\PP$.
\end{theorem}

The properties in Theorem~\ref{ThNonRedundantApex} have a geometric interpretation in terms of the tropical polytope~$\PP$ and the half-space $\HH(a,I)$ which can be associated with the $(I,j)$-pseudovertex $a$. On the one hand, Property~\eqref{ThNonRedundantApexC1} states that $\HH(a,I)$ contains all the generators $v^r$ for $r \in [p]$. Equivalently, this means that $\PP$ is included in $\HH(a,I)$. On the other hand, Properties~\eqref{ThNonRedundantApexC2} and~\eqref{ThNonRedundantApexC3} are related to the tightness of $\HH(a,I)$ w.r.t.\ the polytope $\PP$. More specifically, Property~\eqref{ThNonRedundantApexC2} tells us that every sector $\SS(a,k)$ not contained in $\HH(a,I)$ must contain at least one generator of $\PP$ which also belongs to a sector contained in $\HH(a,I)$. Finally, Property~\eqref{ThNonRedundantApexC3} states that every sector $\SS(a,i)$ contained in $\HH(a,I)$ must contain at least one generator of $\PP$ which also belongs to the special sector $\SS(a,j)$, but not to the other sectors composing $\HH(a,I)$. 

\begin{example}\label{ExampleNonredundantApices}
Using the results of~\cite{AllamigeonKatzJCTA2013}, it can be shown that the non-redundant apices  of the polytope $\PP$ of Figure~\ref{fig:running_example} are the points $a^1=(0,3,3)$, $a^2=(0,5,1)$, $a^3=(0,6,5)$ and $a^4=(0,1,4)$, depicted in green in Figure~\ref{fig:cell_decomposition}. Thus, $\{a^1,a^2, a^3, a^4\}$ is precisely the set of apices of the half-spaces in any non-redundant representation of $\PP$ composed of half-spaces whose apices belong to $\PP$. It follows from Theorem~\ref{ThNonRedundantApex} that these points are $(I,j)$-pseudovertices of $\PP$. For instance, we have $S(a^1)=(\{1,2\},\{2,3\},\{1,4\})$, and the properties of   Theorem~\ref{ThNonRedundantApex} are satisfied for $I=\{2,3\}$ and $j=1$. On the other hand, the points $v^1, \ldots ,v^4$, $u^1$ and $u^2$ of Figure~\ref{fig:cell_decomposition} are pseudovertices of the cell decomposition induced by the generators of $\PP$, but they are not $(I,j)$-pseudovertices of $\PP$. 
\end{example}

\begin{figure}
\begin{center}
\begin{tikzpicture}[convex/.style={draw=lightgray,fill=lightgray,fill opacity=0.7},convexborder/.style={ultra thick},vtx/.style={blue!50},
hsborder/.style={green!60!black,dashdotted,thick},
hs/.style={draw=none,pattern=north east lines,pattern color=green!60!black,fill opacity=0.5},
apex/.style={green!60!black},
shsborder/.style={orange,dashdotted,very thick},
shs/.style={draw=none,pattern=north west lines,pattern color=orange,fill opacity=0.5},
extended line/.style={shorten >=-#1,shorten <=-#1},extended line/.default=1cm,
scale=0.7,>=triangle 45
]

\begin{scope}[xshift=11.9cm]
\draw[gray!30,very thin,step=2] (-3.5,-3.5) grid (3.5,3.5);
\draw[gray!20,very thin,dashed] (-3.5,-3.5) grid (3.5,3.5);
\draw[gray!50,->] (-3.5,0) -- (3.5,0) node[color=gray!50,below,font=\scriptsize] {$x_2-x_1$};
\draw[gray!50,->] (0,-3.5) -- (0,3.5) node[color=gray!50,above,font=\scriptsize] {$x_3-x_1$};

\coordinate (v0) at (-2,2);
\coordinate (v1) at (0,0);
\coordinate (v2) at (2,-2);
\filldraw[convex] (0,0) -- (0,2) -- (2,2) -- (2,0) -- cycle;
\draw[convexborder] (v0) -- (0,2) -- (v1) -- (2,0) -- (v2) -- (2,2) -- cycle;

\filldraw[vtx] (v0) circle (3pt) node[below left=-0.5ex] {$v^1$};
\filldraw[vtx] (v1) circle (3pt) node[below left=-0.5ex] {$v^2$};
\filldraw[vtx] (v2) circle (3pt) node[below left=-0.5ex] {$v^3$};

\draw[shsborder] (-2,-3.5) -- (-2,2) -- (0,4);
\filldraw[shs] (-2,-3.5) -- (-2,2) -- (0,4) -- (0.4,4) -- (-1.7,1.876) -- (-1.7,-3.5) -- cycle;

\end{scope}

\end{tikzpicture}
\end{center}
\caption{
The non-pure tropical polytope generated by the points $v^1=(0,-1,1)$, $v^2=(0,0,0)$ and $v^3=(0,1,-1)$. The point $v^1$ is an $(I,j)$-pseudovertex for $I=\{2\}$ and any $j\in \{1,3\}$. The corresponding half-space, depicted in orange, only contains the generator $v^1$ in its boundary.}\label{fig:non-pure}
\end{figure}
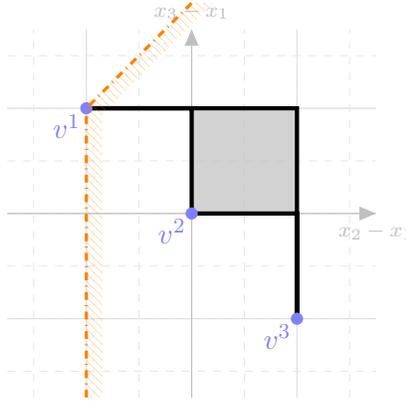

Given a $(I,j)$-pseudovertex $a$ of $\PP$, the tightness properties of Theorem~\ref{ThNonRedundantApex} enforces that $\HH := \HH(a,I)$ is a \emph{minimal} half-space containing $\PP$, meaning that for every half-space $\HH'$, $\PP \subseteq \HH' \subseteq \HH$ implies $\HH' = \HH$ (see~\cite[Proposition~42]{AllamigeonKatzJCTA2013}). As shown in~\cite[Lemma~38]{AllamigeonKatzJCTA2013}, if $\PP$ is pure, there exists at most one minimal half-space with a given apex. Obviously, there is always a non-redundant external representation of $\PP$ composed of minimal half-spaces. By Theorem~\ref{ThNonRedundantApex}, their apices are $(I,j)$-pseudovertices of $\PP$. Therefore, when $\PP$ is pure, every minimal half-space appearing in such an external representation is of the form $\HH(a,I)$, where $a$ is a $(I,j)$-pseudovertex of $\PP$. This leads to the following proposition:
\begin{proposition}\label{PropUniqueExternalRepr}
If the tropical polytope $\PP$ is pure, the set of half-spaces $\HH(a,I)$ associated with the $(I,j)$-pseudovertices $a$ of $\PP$ constitutes an external representation of~$\PP$.
\end{proposition}
 
\begin{example}
Consider again the tropical polytope of Figure~\ref{fig:running_example}. In this case, it can be shown that the $(I,j)$-pseudovertices of $\PP$ coincide with the non-redundant apices of $\PP$. Thus, recalling Example~\ref{ExampleNonredundantApices}, it follows that the external representation provided by Proposition~\ref{PropUniqueExternalRepr} is composed of the half-spaces: $\HH(a^1,\{2,3\})$, $\HH(a^2,\{3\})$, $\HH(a^3,\{1\})$ and $\HH(a^4,\{2\})$.
\end{example}

\begin{remark}\label{remark:superfluous}
It is worth mentioning that except when $n=3$, the external representation given in Proposition~\ref{PropUniqueExternalRepr} may contain superfluous half-spaces. In other words, this may not be a non-redundant external representation.
\end{remark}

As illustrated in Figure~\ref{fig:non-pure}, if the tropical polytope $\PP$ is non-pure, the half-space $\HH(a,I)$ associated with a $(I,j)$-pseudovertex $a$ of $\PP$ may contain strictly less than $n-1$ generators of $\PP$ in its boundary.
In contrast, when $\PP$ is pure, any lift $\liftP$ of $\PP$ is full-dimensional, and so the facets of $\liftP$ contain at least $n-1$ generators of $\liftP$. In such a case, if $\HH(a,I)$ was the image of a facet-defining half-space of $\liftP$ as we want, then $\HH(a,I)$ should contain at least $n-1$ generators of $\PP$ in its boundary. In the next proposition we use Lemma~\ref{LemmaExtTypePure} and the three properties of Theorem~\ref{ThNonRedundantApex} to prove that when $\PP$ is pure, the boundary of $\HH(a,I)$ indeed contains at least $n-1$ generators of $\PP$. 
In fact, we prove that the boundary of $\HH(a,I)$ contains $n-1$ extreme points of $\PP$ satisfying certain properties involving their extremality types. 

\begin{proposition}\label{PropNonRedunApexPure}
Let $a$ be a $(I,j)$-pseudovertex of $\PP$ ($j \in \oneto{n} \setminus I$). If $\PP$ is pure, then the boundary of $\HH(a,I)$ contains at least $n-1$ extreme points of $\PP$.

More precisely, there exists a subset $\{v^{r_k}\}_{k\in \oneto{n} \setminus \{j\}}$ of  $\{v^r\}_{r\in \oneto{p}}$ such that:
\begin{enumerate}[label=(\roman*),ref=\roman*]
\item\label{PropNonRedunApexPureP1} for all $k \in \oneto{n} \setminus (I \cup \{j\})$, $v^{r_k}$ is an extreme point of $\PP$ of type $k$ verifying $r_k \in S_k(a)$; 
\item\label{PropNonRedunApexPureP2}  for all $i \in I$, $v^{r_i}$ is an extreme point of $\PP$ of type $j$ verifying $r_i \in ( S_i(a) \cap S_j(a) ) \setminus \bigl( \bigcup_{h\in I\setminus \{i \}} S_h(a) \bigr)$.
\end{enumerate}
\end{proposition} 

\begin{proof} 
By Property~\eqref{ThNonRedundantApexC2} of Theorem~\ref{ThNonRedundantApex}, for each $k \in \oneto{n} \setminus (I \cup \{j\})$, the set $S_k(a)$ is not empty, and so there exists a generator $v^r$ of $\PP$ in $\SS (a,k)$. Let $v^{r_k}$ be an extreme point of $\PP$ of type $k$ (provided by Remark~\ref{RemarkNeighborhood}) such that $v^{r_k} \in \SS(v^r,k)$. Then, by~\eqref{PropertySectors} we have $v^{r_k} \in \SS(a,k)$, \ie{} 
$r_k \in S_k(a)$.

Now, let $i \in I$. By Property~\eqref{ThNonRedundantApexC3} of Theorem~\ref{ThNonRedundantApex} there exists $r \in \oneto{p}$ such that $r \in S_i(a) \cap S_j(a)$ and $r \not \in \bigcup_{h \in I \setminus \{i\}} S_h(a)$. Moreover, by Remark~\ref{RemarkNeighborhood} we know that there exists $r_i \in \oneto{p}$ such that $v^{r_i}$ is extreme of type $j$ and $v^{r_i} \in \SS(v^r,j)$. It follows that $(\mpinverse{v_j^{r_i}}) \mptimes v_h^{r_i} \leq (\mpinverse{v^r_j}) \mptimes v^r_h \leq (\mpinverse{a}_j) \mptimes a_h$ for all $h \in \oneto{n}$, where the last inequality is strict for $h \in I\setminus \{i\}$. Thus, 
we have $(\mpinverse{a}_h) \mptimes v_h^{r_i}\leq (\mpinverse{a}_j) \mptimes v_j^{r_i}$ 
for $h\in \{i \}\cup (\oneto{n}\setminus I)$ and $(\mpinverse{a}_h) \mptimes v_h^{r_i} < (\mpinverse{a}_j) \mptimes v_j^{r_i}$ for 
$h\in  I\setminus \{i \}$. Besides, since $v^{r_i} \in \PP \subset \HH(a,I)$ (by Property~\eqref{ThNonRedundantApexC1} of Theorem~\ref{ThNonRedundantApex}), there exists $k \in I$ such that $(\mpinverse{a}_k) \mptimes v_k^{r_i} \geq (\mpinverse{a}_j) \mptimes v_j^{r_i}$. We conclude that $k$ must be equal to $i$, and so $(\mpinverse{a}_i) \mptimes v_i^{r_i} = (\mpinverse{a}_j) \mptimes v_j^{r_i} \geq (\mpinverse{a}_h) \mptimes v_h^{r_i}$ for all $h\in \oneto{n}$, with strict inequality for $h\in  I\setminus \{i \}$. In consequence, we have $r_i \in (S_i(a) \cap S_j(a)) \setminus \bigl(\bigcup_{h \in I \setminus \{i\}} S_h(a)\bigr)$.

Observe that for $i, h \in I$, we have $r_i \neq r_{h}$ if $i \neq h$, since $r_i \in S_i(a)\setminus S_{h}(a)$ while $r_{h} \in S_{h}(a)\setminus S_i(a)$. Moreover, by Lemma~\ref{LemmaExtTypePure}, the extreme points $v^{r_k}$ for $k \in \oneto{n} \setminus (I \cup \{j\})$ are pairwise distinct, and distinct from the extreme points $v^{r_i}$ for $i \in I$.

Finally, note that for each $k \in \oneto{n} \setminus \{j\}$ there exists $h\in \oneto{n} \setminus I$ such that $r_k\in S_h(a)$. This together with Property~\eqref{ThNonRedundantApexC1} of Theorem~\ref{ThNonRedundantApex} imply that $v^{r_k}$ belongs to the boundary of $\HH(a,I)$ for all $k \in \oneto{n} \setminus \{j\}$.  
\end{proof}

\begin{example}
Once more, consider the tropical polytope $\PP$ of Figure~\ref{fig:running_example}. For the $(I,j)$-pseudovertex $a^1=(0,3,3)$, we have $S(a^1)=(\{1,2\},\{2,3\},\{1,4\})$, and so $I=\{2,3\}$ and $j=1$. In this case, only Property~\eqref{PropNonRedunApexPureP2} of Proposition~\ref{PropNonRedunApexPure} applies because $\oneto{n} \setminus (I \cup \{j\})=\emptyset$. Observe that $2\in (S_2(a^1)\cap S_1(a^1))\setminus S_3(a^1)$ and $v^2$ is extreme of type $1$. Similarly, we have $1\in (S_3(a^1)\cap S_1(a^1))\setminus S_2(a^1)$ and $v^1$ is extreme of type $1$. Thus, in Proposition~\ref{PropNonRedunApexPure} we have $r_2=2$ and $r_3=1$ for $a=a^1$. 

Now, consider the $(I,j)$-pseudovertex $a^2=(0,5,1)$. Since $S(a^2)=(\{2\},\{3\},\{1,2,3,4\})$, we have $I=\{3\}$, and $j$ can be either $1$ or $2$. Let us take $j=2$, so that $\oneto{n} \setminus (I \cup \{j\})=\{1\}$. Note that $2\in S_1(a^2)$ and $v^2$ is extreme of type $1$. Besides, we have $3\in (S_3(a^2)\cap S_2(a^2))$ and $v^3$ is extreme of type $2$. Thus, in Proposition~\ref{PropNonRedunApexPure} we have $r_1=2$ and $r_3=3$ for $a=a^2$.
\end{example}

\section{Proof of the theorem when the generators are in general position}\label{SectionGeneralP}

In this section we prove the second part of Theorem~\ref{ConjDevelinYu}. With this aim, it is convenient to recall the following results. 

\begin{lemma}[{\cite[Lemma 5.1]{RGST}}]\label{LemmaHyper}
The points $u^1,\dots, u^n \in \tproj^{n-1}$ are contained in a tropical hyperplane if, and only if, the $n\times n$ matrix whose columns are these points is tropically singular.  
\end{lemma}

\begin{theorem}[Corollary of {\cite[Theorem 1.3]{AGG2013}}]\label{TheoremSignHyper}
If the points $u^1,\dots, u^{n-1}\in \tproj^{n-1}$ are in general position, then there exists a unique signed tropical hyperplane which contains these points. 
\end{theorem} 

Recall that when the points $v^1, \dots, v^p$ are in general position and $p \geq n$, the facets of any lift of $\PP$ are simplicial cones, and all lifts are combinatorially equivalent, see~\cite[Proposition~2.3]{DevelinYu}. Before proving the second part of Theorem~\ref{ConjDevelinYu}, we establish the following characterization of the facets of the lifts of $\PP$. 

\begin{proposition}\label{prop:facet_characterization}
Assume the points $v^1, \dots ,v^p$ are in general position, and $p \geq n$. Let $R$ be a subset of $\oneto{p}$ with $n-1$ elements, and $\liftP = \cone \left( \{\liftv^r\}_{r\in \oneto{p}} \right)$ be a lift of $\PP$. 

Then, $\cone (\{\liftv^{r}\}_{r \in R})$ is a facet of $\liftP$ if, and only if, there exists a half-space $\HH(a,I)$ such that $\PP \subset \HH(a,I)$ and the set $\{v^r \}_{r \in R}$ is contained in the boundary of $\HH(a,I)$. When such half-space $\HH(a,I)$ exists, it coincides with the image under the valuation map of the half-space which defines the facet $\cone (\{\liftv^r \}_{r \in R})$.  
\end{proposition}

\begin{proof}
Suppose first that $\cone (\{ \liftv^r \}_{r \in R})$ is a facet of $\liftP$. Let $\liftH = \bigl\{ \liftx \in \K^n\mid \sum_{i\in \oneto{n}} \liftf_i \utimes \liftx_i \geq 0 \bigr\}$ be the corresponding facet-defining half-space, so that it contains $\liftP$, and its boundary contains the set $\{\liftv^r \}_{r \in R}$ (note that $\liftf_i \neq 0$ for all $i\in \oneto{n}$ because the points $\liftv^1, \dots, \liftv^p$ are in general position). By Proposition~\ref{Prop2.4DevelinYu}, the image under the valuation map of $\liftH \cap (\K^+)^n$ is a tropical half-space $\HH(a,I)$, and the image of $\bigl\{ \liftx \in \K^n\mid \sum_{i\in \oneto{n}} \liftf_i \utimes \liftx_i = 0 \bigr\} \cap (\K^+)^n$ is a signed hyperplane, which is equal to the boundary of $\HH(a,I)$. It follows that $\PP \subset \HH(a,I)$, and the set $\{v^r \}_{r \in R}$ is contained in the boundary of $\HH(a,I)$.

Reciprocally, assume that there exists a half-space $\HH(a,I)$ such that $\PP \subset \HH(a,I)$ and the set $\{v^r \}_{r \in R}$ is contained in the boundary of $\HH(a,I)$. Let $\bigl\{ \liftx \in \K^n\mid \sum_{i\in \oneto{n}} \liftf_i \utimes \liftx_i = 0 \bigr\}$ be the unique hyperplane which contains the set $\{\liftv^r \}_{r\in R}$ (again we have $\liftf_i \neq 0$ for all $i\in \oneto{n}$ because the points $\liftv^1, \dots, \liftv^p$ are in general position). Then, by Proposition~\ref{Prop2.4DevelinYu}, the image under the valuation map of $\bigl\{ \liftx \in \K^n\mid \sum_{i\in \oneto{n}} \liftf_i \utimes \liftx_i = 0 \bigr\} \cap (\K^+)^n$ is the signed hyperplane:
\begin{equation}\label{SignedHyperplane}
\bigl\{ x\in \tproj^{n-1}\mid \mpplus_{i\in I'} (\mpinverse{b}_i) \mptimes x_i =
\mpplus_{j\in J'} (\mpinverse{b}_j) \mptimes x_j \bigr\}\; , 
\end{equation} 
where $I' = \{ i\in \oneto{n} \mid \liftf_i>0\}$, 
$J' = \{ j\in \oneto{n} \mid \liftf_j<0 \} = \oneto{n}\setminus I'$ and 
$b_i = \mpinverse{\val (\liftf_i)}$ for $i \in \oneto{n}$. Besides, since $\sum_{i\in \oneto{n}} \liftf_i \utimes \liftv^{r}_i =0$ for $r \in R$, the set $\{v^r \}_{r \in R}$ is contained in this signed hyperplane. Then, by Theorem~\ref{TheoremSignHyper} we deduce that the signed hyperplane~\eqref{SignedHyperplane} coincides with the boundary of $\HH(a,I)$. Thus, we can assume without loss of generality that $I' = I$ (replacing $\liftf$ by $-\liftf$ if necessary), and that $a_i = b_i$ for all $i \in \oneto{n}$.

Let us fix now $r\in\oneto{p}\setminus R$. Since $\PP \subset \HH(a,I)$, we have $\mpplus_{i\in I} (\mpinverse{a}_i) \mptimes v_i^r \geq \mpplus_{j\in \oneto{n} \setminus I} (\mpinverse{a}_j) \mptimes v_j^r$. It follows that $\mpplus_{i\in I} (\mpinverse{a}_i) \mptimes v_i^r > \mpplus_{j\in \oneto{n} \setminus I} (\mpinverse{a}_j) \mptimes v_j^r$ because the points $v^1, \dots ,v^p$ are in general position, and so $v^r$ cannot belong to the boundary of $\HH(a,I)$ by Lemma~\ref{LemmaHyper}. Thus, there exists $k\in I$ such that 
\[
(\mpinverse{a}_k) \mptimes v_k^r > (\mpinverse{a}_j) \mptimes v_j^r \: 
\makebox{ for all } \; j \in \oneto{n} \setminus I \; ,
\]
which implies 
\[
\val (\liftf_k \utimes \liftv_k^r )> \val (\liftf_j \utimes \liftv_j^r) \: 
\makebox{ for all } \; j \in \oneto{n} \setminus I \; . 
\]
It follows that $\sum_{i\in \oneto{n}} \liftf_i \liftv^r_i > 0$  because $\liftf_i>0$ for $i\in I'=I$. Since this holds for any 
$r\in \oneto{p} \setminus R$, 
we conclude that 
$\liftP\subset \bigl\{ \liftx\in \K^n\mid \sum_{i\in \oneto{n}} \liftf_i \utimes \liftx_i \geq 0 \bigr\}$. 

As a consequence, the half-space $\bigl\{ \liftx \in \K^n\mid \sum_{i\in \oneto{n}} \liftf_i \utimes \liftx_i \geq 0 \bigr\}$ defines a facet of $\liftP$. Thanks to Proposition~\ref{Prop2.4DevelinYu}, we know that the image under the valuation map of this half-space is $\HH(a,I)$. This completes the proof.
\end{proof} 

We are now ready to prove the second part of Theorem~\ref{ConjDevelinYu}, exploiting Proposition~\ref{PropUniqueExternalRepr} and the half-spaces associated with the $(I,j)$-pseudovertices.

\begin{theorem}\label{TheoFacetLiftExtRepr}
Assume the points $v^1, \dots ,v^p$ are in general position, and $\PP$ is pure. Then, the images under the valuation map of the facet-defining half-spaces of any lift of $\PP$ provide an external representation of $\PP$ by tropical half-spaces. More precisely, this external representation contains all the half-spaces associated with the $(I,j)$-pseudovertices of $\PP$.
\end{theorem} 

\begin{proof}
Let $\liftP$ be any lift of $\PP$. First of all, recall that the image under the valuation map of any facet-defining half-space of $\liftP$ is a tropical half-space which contains $\PP$. 

Given a $(I,j)$-pseudovertex $a\in \tproj^{n-1}$ of $\PP$ ($j \in \oneto{n} \setminus I$),
by Proposition~\ref{PropNonRedunApexPure} we know that there exist $n-1$ distinct extreme points of $\PP$, $v^{r_k}$ for $k \in \oneto{n} \setminus \{j\}$, contained in the boundary of $\HH(a,I)$. Moreover, as $\PP$ is pure, we necessarily have $p \geq n$ (if $p < n$, we can find a tropical hyperplane containing $\PP$). Then, from Proposition~\ref{prop:facet_characterization} it follows that $\cone \left(\{\liftv^{r_k}\}_{k \in \oneto{n} \setminus \{j\}}\right)$ is a facet of $\liftP$, and that the image under the valuation map of the corresponding facet-defining half-space is $\HH(a,I)$. 

We conclude by using Proposition~\ref{PropUniqueExternalRepr}.
\end{proof}

\begin{remark}
The external representation provided by Theorem~\ref{TheoFacetLiftExtRepr}  
can contain superfluous half-spaces, even in the case where the representation of Proposition~\ref{PropUniqueExternalRepr} is non-redundant. For example, consider the following points of $\tproj^4$:
\[
\begin{aligned}
v^1 & = (0, 2.5176, 10.5161, -0.484, 2.5151)&
v^2 & = (0, 3.5149, 11.0142, 0.0149, 3.0115)\\
v^3 & = (0, 3.0217, 11.5012, 0.0101, 3.0009)&
v^4 & = (0, 3.0154, 11.0145, 0.0056, 3.5239)\\
v^5 & = (0, 2.0238, 14.5216, 13.0094, 13.0252)&
v^6 & = (0, 2.0131, 14.0181, 13.5023, 13.0153)\\
v^7 & = (0, 2.0005, 14.0202, 13.0097, 13.5148)&
v^8 & = (0, 0.5033, 2.5111, 10.5037, 6.5084)\\
v^9 & = (0, 1.5245, 3.001, 11.0068, 7.0053)&
v^{10} & = (0, 1.0232, 3.0155, 11.511, 7.0175)\\
v^{11} & = (0, 1.0083, 3.0047, 11.0005, 7.5248)
\end{aligned}
\]
Note that these points have been numerically perturbed to ensure that they are in general position. The tropical polytope $\PP$ generated by these points is pure and has \emph{generic extremities}. This genericity property, introduced in~\cite[Section~5]{AllamigeonKatzJCTA2013}, ensures in particular that the non-redundant apices of $\PP$ are precisely the $(I,j)$-pseudovertices of $\PP$, see~\cite[Theorem~51]{AllamigeonKatzJCTA2013}. This allows us to verify that the external representation of Proposition~\ref{PropUniqueExternalRepr}, consisting of $26$ half-spaces, is non-redundant. In contrast, the number of facets of any lift $\liftP$ of $\PP$, which can be determined thanks to Proposition~\ref{prop:facet_characterization}, is equal to $27$. More precisely, the points $\liftv^1$, $\liftv^3$, $\liftv^5$, and $\liftv^8$ define a facet of $\liftP$, and the image under the valuation map of the corresponding facet-defining half-space is $\HH(a,\{2,4,5\})$, where $a = (0,2.5176,10.9971,10.5037,9.5007)$. The type of $a$ is 
\[
S(a) = \big(\{ 1, 8 \}, \{ 1, 2, 3, 4\}, \{ 3, 5 \}, \{ 8,9,10,11\}, \{5,6,7\}\big) \ .
\]
If this point was a $(I,j)$-pseudovertex of $\PP$, then by Property~\eqref{ThNonRedundantApexC1} of Theorem~\ref{ThNonRedundantApex} we would necessarily have $I \supset \{2,4,5\}$. However, observe that for any $j \in \{1,3\}$, there exists $i \in \{2,4,5\}$ which falsifies Property~\eqref{ThNonRedundantApexC3} of Theorem~\ref{ThNonRedundantApex}. 
We deduce that the half-space $\HH(a,\{2,4,5\})$ is superfluous in the external representation of $\PP$ provided by Theorem~\ref{TheoFacetLiftExtRepr}.
\end{remark}

\section{Proof of the theorem when the generators are in arbitrary position}\label{SectionArbitraryP}

In this section we consider the case in which $\PP$ is pure but its generators $v^1, \dots, v^p$ are not necessarily in general position. In consequence, the lifts of $\PP$ may not be combinatorially equivalent anymore. 

As announced earlier, our proof of the first part of Theorem~\ref{ConjDevelinYu} relies on a symbolic perturbation technique. With this aim, we consider the tropical semiring associated with the group $(\R^3, +, \leqlex )$, as defined in Section~\ref{SubSectionTG}. The relation $\leqlex$ refers to the lexicographic ordering over triples, and the group law $+$ is the component-wise addition. Intuitively, the triple $(\lambda_1, \lambda_2, \lambda_3) \in \R^3$ represents the element $\lambda_1 + \lambda_2 \epsilon + \lambda_3 \epsilon^2$, where $\epsilon > 0$ plays the role of an infinitesimal, and $\lambda_2$ and $\lambda_3$ stand for a first-order and a second-order perturbation respectively. 
For the sake of simplicity, henceforth we denote by $\Tp$ the tropical semiring associated with $(\R^3, +, \leqlex )$, by $\tprojp{}^{n-1}$ the corresponding projective space, by $\Kp$ the field of real Hahn series with value group $(\R^3, +, \leqlex )$ and by $\Kp^+$ the set of positive elements of $\Kp$. Given $\xp = (\xp_1, \xp_2, \xp_3) \in \Tp$, we denote by $\pi_i(\xp) := \xp_i$ the projection on the $i$-th coordinate. We extend this notation to $\Tp^n$ and $\tprojp{}^{n-1}$ in a coordinate-wise way.

Our approach relies on the application of Theorem~\ref{TheoFacetLiftExtRepr} to a tropical polytope of $\tprojp{}^{n-1}$ encoding a perturbation of $\PP$. To this aim, we need to make sure that Theorem~\ref{TheoFacetLiftExtRepr} still holds in the perturbed setting, \ie~when substituting $\T$ by $\Tp$ and $\K$ by $\Kp$. The proof of Theorem~\ref{TheoFacetLiftExtRepr} and of its main ingredient, Proposition~\ref{prop:facet_characterization}, involve two kinds of results. On the one hand, it is based on the connection between ordinary polyhedral sets and their tropical counterparts through the valuation map, presented in Section~\ref{SubSectionLifts}. It can be verified that these results remain valid in the perturbed setting. In more details, their proof uses the fact that the valuation map from $\Kp^+ \cup \{0\}$ to $\Tp$ is a monotone (semiring) homomorphism. In Proposition~\ref{Prop2.4DevelinYu}, we also use the fact that if $\bot < x \leq y$, then we can find two positive elements $\liftx$ and $\lifty$ such that $\deg(\liftx) = x$, $\deg(\lifty) = y$ and $\liftx < \lifty$. On the other hand, the proof of Theorem~\ref{TheoFacetLiftExtRepr} depends on several properties of tropical objects (polytopes, hyperplanes and half-spaces), namely Lemmas~\ref{LemmaExtTypePure} and~\ref{LemmaHyper}, Theorems~\ref{ThNonRedundantApex} and~\ref{TheoremSignHyper}, and Propositions~\ref{CharactExtP}, \ref{PropUniqueExternalRepr} and~\ref{PropNonRedunApexPure}. We propose to use arguments from model theory to ensure that these results are still correct when replacing $\T$ by $\Tp$. Indeed, their statements can be expressed as first-order formul{\ae} in the language of ordered groups.\footnote{The language of ordered groups is given by $\mathcal{L} = \{0, +, -, <\}$, where $0$ is a constant standing for the neutral element, $+$ and $-$ are respectively binary and unary functions representing the addition and its inverse, and $<$ is a binary relation representing the (proper) ordering.} The groups $(\R, +, \leq)$ and $(\R^3, +, \leqlex)$ are both models of the theory of non-trivial ordered divisible abelian groups, \ie~they satisfy the axioms defining the theory of such groups (see~\cite{Marker} for an introduction to these concepts). As shown by Marker in~\cite[Corollary~3.1.17]{Marker}, the models of this theory are elementarily equivalent, meaning that a formula is valid in a model of the theory if and only if it is valid in any of its models. In other words, the aforementioned results hold in $\Tp$, based on the fact that they are valid in $\T$. Following this discussion, it can be checked that the arguments used in the proof of Proposition~\ref{prop:facet_characterization} and Theorem~\ref{TheoFacetLiftExtRepr} are still correct in the setting of $\Tp$. We deduce that Theorem~\ref{TheoFacetLiftExtRepr} can be applied to tropical polytopes of $\tprojp{}^{n-1}$ and their lifts over $\Kp^n$.

We consider a symbolic perturbation of the generators $v^1, \dots, v^p$ of $\PP$ given by the following points of $\tprojp{}^{n-1}$:
\begin{equation}\label{DefPerturbedPoint}
\vp^r   \mydef 
\begin{cases}
(v^r,e^i,\gamma^r) & \text{if} \ v^r \ \text{is an extreme point of} \ \PP \ \text{of type} \ i  , \\
(v^r,0,\gamma^r) & \text{otherwise,} \\
\end{cases} 
\end{equation}
for $r \in \oneto{p}$, where $e^i \in \R^n$ is the $i$-th element of the canonical basis of $\R^n$, $0$ is the origin of $\R^n$, and $\gamma^r$ is a point of $\R^n$. Note that the points $\vp^r$ are well-defined because each extreme point of $\PP$ has a unique extremality type (by Lemma~\ref{LemmaExtTypePure}). The points $\gamma^1, \dots, \gamma^p $ are chosen to be in general position. Thanks to this assumption, the symbolically perturbed points $\vp^1, \dots, \vp^p$ are also in general position. We define $\PPp \subset \tprojp{}^{n-1}$ as the tropical polytope generated by the points $\vp^1,\ldots ,\vp^p$. This provides a perturbation of $\PP$, in the sense that $\pi_1(\PPp)=\PP$, since $\pi_1(\vp^r)=v^r$ for all $r \in \oneto{p}$. This perturbation scheme is illustrated in Figures~\ref{fig:perturbation}(a) and~\ref{fig:perturbation}(b).

\begin{figure}
\begin{center}
\begin{tikzpicture}[convex/.style={draw=lightgray,fill=lightgray,fill opacity=0.7},convexborder/.style={ultra thick},vtx/.style={blue!50},
extended line/.style={shorten >=-#1,shorten <=-#1},extended line/.default=1cm,
scale=0.9,>=triangle 45
]
\begin{scope}
\draw[gray!30,very thin,step=2] (-1.5,-1.5) grid (3.5,3.5);
\draw[gray!20,very thin,dashed] (-1.5,-1.5) grid (3.5,3.5);
\draw[gray!50,->] (-1.5,0) -- (3.5,0) node[color=gray!50,below,font=\scriptsize] {$x_2-x_1$};
\draw[gray!50,->] (0,-1.5) -- (0,3.5) node[color=gray!50,above,font=\scriptsize] {$x_3-x_1$};
\filldraw[convex] (0,0) -- (0,2) -- (2,2) -- (2,0)-- cycle;
\draw[convexborder] (0,0) -- (0,2) -- (2,2) -- (2,0)-- cycle;
\filldraw[vtx] (0,0) circle (3pt) node[below left=-0.5ex] {$v^1$};
\filldraw[vtx] (0,2) circle (3pt) node[above left=-0.5ex] {$v^3$};
\filldraw[vtx] (2,0) circle (3pt) node[below right=-0.5ex] {$v^2$};

\node[font=\normalfont] at (1,-2) {(a)};
\end{scope}

\begin{scope}[shift={(8,0)}]
\draw[gray!30,very thin,step=2] (-1.5,-1.5) grid (3.5,3.5);
\draw[gray!20,very thin,dashed] (-1.5,-1.5) grid (3.5,3.5);
\draw[gray!50,->] (-1.5,0) -- (3.5,0) node[color=gray!50,below,font=\scriptsize] {$x_2-x_1$};
\draw[gray!50,->] (0,-1.5) -- (0,3.5) node[color=gray!50,above,font=\scriptsize] {$x_3-x_1$};

\filldraw[convex] (-0.3,-0.3) -- (-0.3,2) -- (0,2.3) -- (2.3,2.3) -- (2.3,0) -- (2,-0.3) -- cycle;
\draw[convexborder] (-0.3,-0.3) -- (-0.3,2) -- (0,2.3) -- (2.3,2.3) -- (2.3,0) -- (2,-0.3) -- cycle;

\filldraw[vtx] (-0.3,-0.3) circle (3pt) node[below left=-0.5ex] {$\vp^1$};
\filldraw[vtx] (0,2.3) circle (3pt) node[above left=-0.5ex] {$\vp^3$};
\filldraw[vtx] (2.3,0) circle (3pt) node[below right=-0.5ex] {$\vp^2$};

\node[font=\normalfont] at (1,-2) {(b)};
\end{scope}

\begin{scope}[shift={(8,-7)}]
\draw[gray!30,very thin,step=2] (-1.5,-1.5) grid (3.5,3.5);
\draw[gray!20,very thin,dashed] (-1.5,-1.5) grid (3.5,3.5);
\draw[gray!50,->] (-1.5,0) -- (3.5,0) node[color=gray!50,below,font=\scriptsize] {$x_2-x_1$};
\draw[gray!50,->] (0,-1.5) -- (0,3.5) node[color=gray!50,above,font=\scriptsize] {$x_3-x_1$};

\filldraw[convex,fill opacity=0.5] (-0.3,-0.3) -- (-0.3,2) -- (0,2.3) -- (2.3,2.3) -- (2.3,0) -- (2,-0.3) -- cycle;

\begin{scope}
\clip 	(-1.5,-1.5) rectangle (3.5,3.5);

\begin{scope}[shift={(-0.3,2)},shsborder/.style={green!60!black,dashdotted,very thick},
shs/.style={draw=none,pattern=north west lines,pattern color=green!60!black,fill opacity=0.5}]
\draw[shsborder] (0,-10) -- (0,0) -- (10,10);
\filldraw[shs] (0,-10) -- (0,0) -- (10,10) -- (10.29,10) -- (0.2,-0.09) -- (0.2,-10) -- cycle;
\filldraw[green!60!black] (0,0) circle (3pt) node[below left=-0.5ex] {$\ap^2$};
\end{scope}

\begin{scope}[shift={(2,-0.3)},shsborder/.style={red!50,dashdotted,very thick},
shs/.style={draw=none,pattern=north west lines,pattern color=red!50,fill opacity=0.5}]
\draw[shsborder] (-10,0) -- (0,0) -- (10,10);
\filldraw[shs] (-10,0) -- (0,0) -- (10,10) -- (10,10.29) -- (-0.09,0.2) -- (-10,0.2) -- cycle;
\filldraw[red!50] (0,0) circle (3pt) node[below left=-0.5ex] {$\ap^3$};
\end{scope}

\begin{scope}[shift={(2.3,2.3)},shsborder/.style={orange,dashdotted,very thick},
shs/.style={draw=none,pattern=north east lines,pattern color=orange,fill opacity=0.5}]
\draw[shsborder] (-10,0) -- (0,0) -- (0,-10);
\filldraw[shs] (-10,0) -- (0,0) -- (0,-10) -- (-0.2,-10) -- (-0.2,-0.2) -- (-10,-0.2) -- cycle;
\filldraw[orange] (0,0) circle (3pt) node[above right=-0.1ex] {$\ap^1$};
\end{scope}
\end{scope}

\node[font=\normalfont] at (1,-2) {(c)};
\end{scope}

\begin{scope}[shift={(0,-7)}]
\draw[gray!30,very thin,step=2] (-1.5,-1.5) grid (3.5,3.5);
\draw[gray!20,very thin,dashed] (-1.5,-1.5) grid (3.5,3.5);
\draw[gray!50,->] (-1.5,0) -- (3.5,0) node[color=gray!50,below,font=\scriptsize] {$x_2-x_1$};
\draw[gray!50,->] (0,-1.5) -- (0,3.5) node[color=gray!50,above,font=\scriptsize] {$x_3-x_1$};

\filldraw[convex,fill opacity=0.5] (0,0) -- (0,2) -- (2,2) -- (2,0)-- cycle;

\begin{scope}
\clip 	(-1.5,-1.5) rectangle (3.5,3.5);

\begin{scope}[shift={(0,2)},shsborder/.style={green!60!black,dashdotted,very thick},
shs/.style={draw=none,pattern=north west lines,pattern color=green!60!black,fill opacity=0.5}]
\draw[shsborder] (0,-10) -- (0,0) -- (10,10);
\filldraw[shs] (0,-10) -- (0,0) -- (10,10) -- (10.29,10) -- (0.2,-0.09) -- (0.2,-10) -- cycle;
\filldraw[green!60!black] (0,0) circle (3pt) node[below left=-0.5ex] {$\pi_1(\ap^2)$};
\end{scope}

\begin{scope}[shift={(2,0)},shsborder/.style={red!50,dashdotted,very thick},
shs/.style={draw=none,pattern=north west lines,pattern color=red!50,fill opacity=0.5}]
\draw[shsborder] (-10,0) -- (0,0) -- (10,10);
\filldraw[shs] (-10,0) -- (0,0) -- (10,10) -- (10,10.29) -- (-0.09,0.2) -- (-10,0.2) -- cycle;
\filldraw[red!50] (0,0) circle (3pt) node[below left=-0.5ex] {$\pi_1(\ap^3)$};
\end{scope}

\begin{scope}[shift={(2,2)},shsborder/.style={orange,dashdotted,very thick},
shs/.style={draw=none,pattern=north east lines,pattern color=orange,fill opacity=0.5}]
\draw[shsborder] (-10,0) -- (0,0) -- (0,-10);
\filldraw[shs] (-10,0) -- (0,0) -- (0,-10) -- (-0.2,-10) -- (-0.2,-0.2) -- (-10,-0.2) -- cycle;
\filldraw[orange] (0,0) circle (3pt) node[above right=-0.7ex] {$\pi_1(\ap^1)$};
\end{scope}
\end{scope}

\node[font=\normalfont] at (1,-2) {(d)};
\end{scope}

\end{tikzpicture}
\end{center}
\caption{(a)~The pure tropical polytope $\PP$ generated by the points $v^1 = (0,0,0)$, $v^2 = (0,1,0)$ and $v^3 = (0,0,1)$. These points are extreme of type $1$, $2$ and $3$ respectively; (b)~A representation of the polytope $\PPp$, neglecting the second-order perturbation, \ie~any element $(\xp_1, \xp_2, \xp_3)$ of $\Tp$ is represented by $\xp_1 + \epsilon \xp_2$ where $\epsilon > 0$ is fixed; (c)~The half-spaces associated with the $(I,j)$-pseudovertices of $\PPp$; (d)~The projection under $\pi_1$ of the previous half-spaces.}\label{fig:perturbation}
\end{figure}

The first-order perturbation that we have made ensures that the extremality types, and their uniqueness, are preserved:
\begin{lemma}\label{lemma:perturbed_extremality_type}
For all $r \in \oneto{p}$ and $i\in \oneto{n}$, the following equivalence holds: $v^r$ is an extreme point of $\PP$ of type $i$ if, and only if, $\vp^r$ is an extreme point of $\PPp$ of type $i$.
\end{lemma}

\begin{proof}
First, suppose that $v^r$ is not an extreme point of $\PP$ of type $i$. By Remark~\ref{RemarkNeighborhood} we know that there exists $s \in \oneto{p}$ such that $v^s \in \SS(v^r,i)$ and $v^s$ is an extreme point of $\PP$ of type $i$. Then $s\neq r$, and by~\eqref{DefPerturbedPoint} for any $k \in \oneto{n}\setminus \{i\}$ we have $(\mpinverse{\vp^r_k})\mptimes \vp^s_k \leqlex ((\mpinverse{v^r_k}) \mptimes v^s_k, 0, (\mpinverse{\gamma^r_k}) \mptimes \gamma^s_k)$, while $(\mpinverse{\vp^r_i}) \mptimes \vp^s_i =  ((\mpinverse{v^r_i}) \mptimes v^s_i, 1, (\mpinverse{\gamma^r_i}) \mptimes \gamma^s_i)$. Since $v^s \in \SS(v^r,i)$, \ie{} $(\mpinverse{v^r_k})\mptimes v^s_k \leq (\mpinverse{v^r_i}) \mptimes v^s_i$ for all $k\in \oneto{n}$, it follows that $\vp^s \in \SS(\vp^r,i)$, and so $\vp^r$ is not an extreme point of $\PPp$ of type $i$.

Now, suppose that $\vp^r$ is not an extreme point of $\PPp$ of type $i$. Let $\vp^s$ be the extreme point of $\PPp$ of type $i$ provided by Remark~\ref{RemarkNeighborhood} which satisfies $\vp^s \in \SS(\vp^r,i)$. Then, we have $s\neq r$ and $v^s \in \SS(v^r,i)$, and so $v^r$ is not an extreme point of $\PP$ of type $i$.
\end{proof}

Thanks to Lemma~\ref{LemmaExtTypePure}, we obtain:

\begin{corollary}\label{CoroPurePerturbed}
The tropical polytope $\PPp$ is pure.	
\end{corollary}

The choice of the first-order perturbation is also motivated by the following lemma. It will allow us to show in Proposition~\ref{prop:projection_halfspaces} that an external representation of $\PP$ can be obtained by ``projecting'' under the map $\pi_1$ the external representation of $\PPp$ associated with its $(I,j)$-pseudovertices. We refer to Figures~\ref{fig:perturbation}(c) and~\ref{fig:perturbation}(d) for an illustration.

\begin{lemma}\label{lemma:fst_order_pert_apex}
Let $\ap$ be a $(I,j)$-pseudovertex of $\PPp$. Then, up to the (tropical) multiplication of $\ap$ by a scalar, the vector $\pi_2(\ap) \in \tproj^{n-1}$ is given by $(\pi_2(\ap))_i = 0$ if $i \in I$, and $(\pi_2(\ap))_i = 1$ otherwise.
\end{lemma}

\begin{proof}
By abuse of notation, in this proof we denote by $S(\xp)$ the type of $\xp \in \tprojp{}^{n-1}$ with respect to the points $\vp^1, \dots, \vp^p$. Up to multiplying $\ap$ by a scalar, we can assume that $(\pi_2(\ap))_j = 1$. 

In the first place, let us consider $i \in I$. By Property~\eqref{PropNonRedunApexPureP2} of Proposition~\ref{PropNonRedunApexPure} applied to $\PPp$, we know that there exists  $r_i \in S_i(\ap) \cap S_j(\ap)$ such that  $\vp^{r_i}$ is an extreme point of $\PPp$ of type~$j$. Then, Lemma~\ref{lemma:perturbed_extremality_type} ensures that $v^{r_i}$ is an extreme point of $\PP$ of type $j$, and so $\vp^{r_i}$ is of the form $(v^{r_i}, e^j, \gamma^{r_i})$. Thus, in particular we have $\pi_2(\vp^{r_i}_i) = 0$ and $\pi_2(\vp^{r_i}_j) = 1$. Since $(\mpinverse{\ap_i}) \mptimes \vp^{r_i}_i = (\mpinverse{\ap_j}) \mptimes \vp^{r_i}_j$ (due to the fact that $r_i \in S_i(\ap) \cap S_j(\ap)$) and $(\pi_2(\ap))_j = 1$, we deduce that $(\pi_2(\ap))_i = 0$. 

Similarly, given $k \in \oneto{n} \setminus (I \cup \{j\})$, by Property~\eqref{PropNonRedunApexPureP1} of Proposition~\ref{PropNonRedunApexPure}  applied to $\PPp$ there exists an extreme point $\vp^{r_k}$ of $\PPp$ of type $k$ such that $r_k \in S_k(\ap)$. By Property~\eqref{ThNonRedundantApexC1} of Theorem~\ref{ThNonRedundantApex}, there exists $i \in I$ such that $r_k \in S_i(\ap)$. This leads to the equality $(\mpinverse{\ap_i}) \mptimes \vp^{r_k}_i = (\mpinverse{\ap_k}) \mptimes \vp^{r_k}_k$. Besides, by Lemma~\ref{lemma:perturbed_extremality_type} we deduce that $\pi_2(\vp^{r_k}_i) = 0$ and $\pi_2(\vp^{r_k}_k) = 1$. From this and the fact that $(\pi_2(\ap))_i = 0$ (shown in the previous paragraph), we conclude that $(\pi_2(\ap))_k = 1$.
\end{proof}

By Corollary~\ref{CoroPurePerturbed}, we can apply Proposition~\ref{PropUniqueExternalRepr} to the tropical polytope $\PPp$, and so  the half-spaces associated with the $(I, j)$-pseudovertices of $\PPp$ constitute an external representation of $\PPp$. It is also useful to observe that given $\ap \in \tprojp{}^{n-1}$ and $I \subset \oneto{n}$, the projection $\pi_1(\HH(\ap,I))$ coincides with the half-space $\HH(\pi_1(\ap), I)$.

\begin{proposition}\label{prop:projection_halfspaces}
The images under $\pi_1$ of the half-spaces associated with the $(I, j)$-pseudovertices of $\PPp$ provide an external representation of $\PP$. 
\end{proposition}

\begin{proof}
Let $\ap$ be a $(I, j)$-pseudovertex of $\PPp$. We claim that $x \in \pi_1(\HH(\ap, I))$ implies $(x, 0, 0) \in \HH(\ap, I)$. Indeed, if $x \in \pi_1(\HH(\ap, I))$, then (as we observed above) $x \in \HH(\pi_1(\ap), I)$, and so there exists $i \in I$ such that $(\mpinverse{\pi_1(\ap)})_j \mptimes x_j \leq  (\mpinverse{\pi_1(\ap)})_i \mptimes x_i$ for all $j \in \oneto{n}\setminus I$. Since by Lemma~\ref{lemma:fst_order_pert_apex} $(\mpinverse{\pi_2(\ap)})_j < (\mpinverse{\pi_2(\ap)})_i$ for any $j \in \oneto{n}\setminus I$, we deduce that $(-\ap_j) \mptimes (x_j, 0, 0) \leqlex (-\ap_i) \mptimes (x_i, 0, 0)$ for all $j \in \oneto{n}\setminus I$, and so $(x, 0, 0) \in \HH(\ap, I)$. This proves our claim. 

As a result of the previous paragraph, if $x$ belongs to the intersection of the images under $\pi_1$ of the half-spaces associated with the $(I, j)$-pseudovertices of $\PPp$, then $(x, 0, 0)$ lies in the intersection of these half-spaces, and so in $\PPp$ by Proposition~\ref{PropUniqueExternalRepr}, which ensures $x \in \PP$ since $\PP = \pi_1(\PPp)$. 

Reciprocally, $\PP$ is obviously included in the intersection of the images under $\pi_1$ of the half-spaces associated with the $(I, j)$-pseudovertices of $\PPp$ because each of these half-spaces contains $\PPp$ and $\PP = \pi_1(\PPp)$.
\end{proof}

Now, let us consider the lift $\liftPp$ of $\PPp$ over $\Kp^n$ defined as the cone generated by the vectors $\liftvp^1, \ldots ,\liftvp^p$, where 
\begin{equation}\label{DefLiftPerturbedPoint}
\liftvp^r_i \mydef t^{-\vp^r_i} 
\end{equation} 
for all $i \in \oneto{n}$ and $r\in \oneto{p}$. By Theorem~\ref{TheoFacetLiftExtRepr}, we know that the set composed of the images under the valuation map of the facet-defining half-spaces of $\liftPp$ contains all the half-spaces associated with the $(I, j)$-pseudovertices of $\PPp$. We are going to build a lift $\liftP$ of $\PP$ based on the lift $\liftPp$ which will prove the first part of Theorem~\ref{ConjDevelinYu}. 

Given real numbers $\mu$ and $\nu$ satisfying $0 < \mu \ll \nu \ll 1$, we define $\liftP$ as the cone generated by the vectors $\liftv^1,\ldots ,\liftv^p$, where
\begin{equation}\label{eq:lift_def}
\liftv^r_i  \mydef 
\bigl(\mu^{-\pi_2(\vp^r_i)} {\nu}^{-\pi_3(\vp^r_i)} \bigr) t^{-\pi_1(\vp^r_i)}  
\end{equation}
for all $i \in \oneto{n}$ and $r\in \oneto{p}$. Observe that $\val(\liftv^r)=\pi_1(\vp^r)=v^r$ for all $r\in \oneto{p}$, and so $\liftP$ is a lift of $\PP$, as desired.

The next result will allow us to relate $\liftPp$ to $\liftP$.

\begin{lemma}\label{lemma:projection}
Let $\liftxp = \sum_{\alpha \in \R^3} \xp_\alpha t^\alpha$ be an element of $\Kp$ with finite support, and let $\mu$ and $\nu$ be real numbers satisfying $0 < \mu \ll \nu \ll 1$. If we define $\liftx \mydef \sum_{\alpha \in \R^3} (\xp_\alpha \mu^{\pi_2(\alpha)} \nu^{\pi_3(\alpha)}) t^{\pi_1(\alpha)} \in \K$,  then 
\[
\val(\liftx) = \pi_1(\val(\liftxp)) \, ,
\]
and the signs of $\liftx$ and $\liftxp$ are identical.
\end{lemma}

\begin{proof}
Let $\Ap \subset \R^3$ be the support of $\liftxp$. 
Defining $A\mydef \pi_1 (\Ap)$, we can write $\liftx = \sum_{\beta \in A} d_\beta t^\beta$ and $\liftxp = \sum_{\beta \in A} \liftdp_\beta t^{(\beta, 0, 0)}$, where the $d_\beta$ and $\liftdp_\beta$ are respectively given by:
\begin{align*}
d_\beta & = \sum_{\alpha \in \Ap \cap \pi^{-1}_1(\beta)} \xp_\alpha \mu^{\pi_2(\alpha)} \nu^{\pi_3(\alpha)} \\
\liftdp_\beta & =  \sum_{\alpha \in \Ap \cap \pi^{-1}_1(\beta)} \xp_\alpha t^{(0, \pi_2(\alpha), \pi_3(\alpha))}
\end{align*}
Provided that $\nu$ is sufficiently small, and that $\mu$ is sufficiently small with respect to $\nu$, $d_\beta \in \R$ and $\liftdp_\beta \in \Kp$ have the same sign. The result now follows from the fact that the sign of $\liftxp$ is given by the sign of $\liftdp_{\pi_1(\mpinverse{\val (\liftxp ))}}$, and that we obviously have $\val(\liftx) = \pi_1(\val(\liftxp))$ by the definition of $\liftx$.
\end{proof}

We are now ready to provide a proof of the first part of Theorem~\ref{ConjDevelinYu}, which is given by the next proposition. 

\begin{proposition}
The images under the valuation map of the facet-defining half-spaces of the lift $\liftP$ defined above provide an external representation of $\PP$. 
\end{proposition}

\begin{proof}
By Lemma~\ref{lemma:projection}, each minor of the matrix with columns $\liftv^1, \dots, \liftv^p$ has the same sign as the corresponding minor of the matrix with columns $\liftvp^1, \dots , \liftvp^p$. Besides, the valuation of the former is given by the image under $\pi_1$ of the valuation of the latter. In consequence, the polyhedra $\liftP$ and $\liftPp$ are combinatorially equivalent. Moreover, if $\liftH = \bigl\{\liftx \in \K^n \mid \sum_{i\in \oneto{n}} \liftf_i \liftx_i \geq 0 \bigr\}$ is the defining half-space of a facet of $\liftP$, and $\liftHp = \bigl\{ \liftxp \in \Kp^n \mid \sum_{i\in \oneto{n}} \liftfp_i \liftxp_i \geq 0 \bigr\}$ is the defining half-space of the corresponding facet of $\liftPp$, then the vectors $\val(\liftf)$ and $\pi_1(\val(\liftfp))$ coincide in $\tproj^{n-1}$. Indeed, up to a multiplicative factor applied to $\liftf$ and $\liftfp$, each coefficient $\liftf_i$ is given by a minor of the matrix with columns $\liftv^1, \dots, \liftv^p$, and $\liftfp_i$ is given by the corresponding minor of the matrix with columns $\liftvp^1, \dots , \liftvp^p$. Since by Proposition~\ref{Prop2.4DevelinYu} we have 
$\val(\liftH \cap \K^+)=\HH (a,I)$ and 
$\val(\liftHp \cap \Kp^+)=\HH (\ap,I)$,
 where $a=(\mpinverse{\val (\liftf_1)}, \ldots ,\mpinverse{\val (\liftf_n)})$, $\ap=(\mpinverse{\val (\liftfp_1)}, \ldots ,\mpinverse{\val (\liftfp_n)})$, and $I=\{i\in \oneto{n}\mid \liftf_i > 0\}=\{i\in \oneto{n}\mid \liftfp_i > 0\}$, it follows that $\pi_1 (\val(\liftHp \cap \Kp^+))= \pi_1(\HH (\ap,I))=\HH (\pi_1(\ap),I)=\HH (a,I)=\val(\liftH \cap \K^+)$.

By Theorem~\ref{TheoFacetLiftExtRepr}, we know that the set composed of the images under the valuation map of the facet-defining half-spaces of $\liftPp$ is an external representation of $\PPp$ which contains all the half-spaces associated with the $(I, j)$-pseudovertices of $\PPp$. Besides, by Proposition~\ref{prop:projection_halfspaces}, the images under $\pi_1$ of these half-spaces form an external representation of $\PP$. Finally, the image of any facet-definining half-space of $\liftP$ under the valuation map contains $\PP$. We deduce from the previous paragraph that the images under the valuation map of the facet-defining half-spaces of $\liftP$ yield an external representation of $\PP$. 
\end{proof}

We conclude the paper with a remark concerning Theorem~\ref{ConjDevelinYu}.

\begin{remark}
The first part of Theorem~\ref{ConjDevelinYu} does not apply to all lifts of $\PP$ when its generators are in arbitrary position. For instance, consider the tropical polytope $\PP$ in Figure~\ref{fig:perturbation}(a) generated by the points $(0,0,0)$, $(0,1,0)$ and $(0,0,1)$, and the lift~$\liftP$ generated by $\lift{w}^1 = (1, 1, 1)$,  $\lift{w}^2 = (1, t^{-1}, \delta)$ and $\lift{w}^3 = (1, \delta, t^{-1})$, where the coefficient $\delta$ belongs to $]0,1[$. As illustrated in Figure~\ref{fig:tropicalization}, it can be verified that the intersection of the images under the valuation map of the facet-defining half-spaces of this lift strictly contains the initial polytope. 

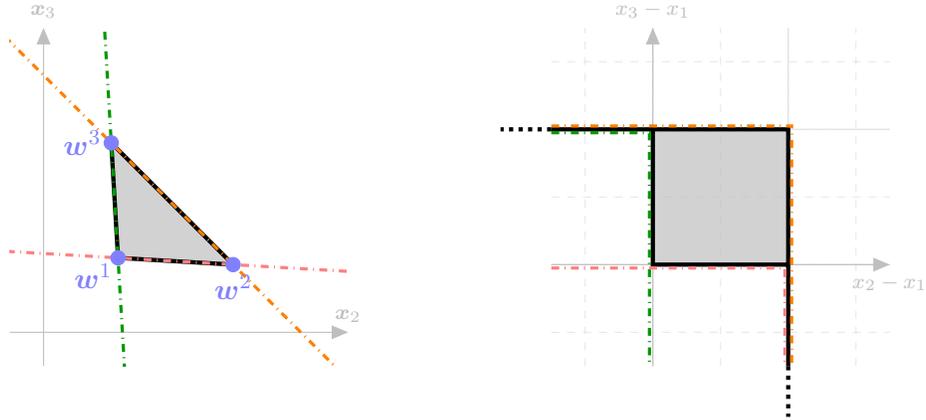
\begin{figure}
\begin{center}
\begin{tikzpicture}[convex/.style={draw=lightgray,fill=lightgray,fill opacity=0.7},convexborder/.style={ultra thick},vtx/.style={blue!50},
shsborder/.style={orange,dashdotted,very thick},
shs/.style={draw=none,pattern=north west lines,pattern color=orange,fill opacity=0.5},
extended line/.style={shorten >=-#1,shorten <=-#1},extended line/.default=1cm,
scale=0.9,>=triangle 45
]

\begin{scope}[shift={(0,0)}]
\coordinate (v0) at (0.1,0.1);
\coordinate (v1) at (0,1.8);
\coordinate (v2) at (1.8,0);

\draw[gray!50,->] (-1.5,-1) -- (3.5,-1) node[color=gray!50,above,font=\scriptsize] {$\lift{x}_2$};
\draw[gray!50,->] (-1,-1.5) -- (-1,3.5) node[color=gray!50,above,font=\scriptsize] {$\lift{x}_3$};

\filldraw[convex] (v0) -- (v1) -- (v2) -- cycle;
\draw[convexborder] (v0) -- (v1) -- (v2) -- cycle;

\begin{scope}
\clip (-1.5,-1.5) rectangle (3.5,3.5);

\draw[shsborder,green!60!black,extended line=2cm] (v1) -- (v0);
\draw[shsborder,red!50,extended line=2cm] (v0) -- (v2);
\draw[shsborder,extended line=2cm] (v1) -- (v2);
\end{scope}

\filldraw[vtx] (v0) circle (3pt) node[below left=-0.5ex] {$\lift{w}^1$};
\filldraw[vtx] (v1) circle (3pt) node[left] {$\lift{w}^3$};
\filldraw[vtx] (v2) circle (3pt) node[below] {$\lift{w}^2$};
\end{scope}

\begin{scope}[shift={(8,0)}]

\draw[gray!30,very thin,step=2] (-1.5,-1.5) grid (3.5,3.5);
\draw[gray!20,very thin,dashed] (-1.5,-1.5) grid (3.5,3.5);
\draw[gray!50,->] (-1.5,0) -- (3.5,0) node[color=gray!50,below,font=\scriptsize] {$x_2-x_1$};
\draw[gray!50,->] (0,-1.5) -- (0,3.5) node[color=gray!50,above,font=\scriptsize] {$x_3-x_1$};
\filldraw[convex] (0,0) -- (0,2) -- (2,2) -- (2,0) -- cycle;
\draw[convexborder] (0,0) -- (0,2) -- (2,2) -- (2,0) -- cycle (0,2) -- (-1.5,2) (2,0) -- (2,-1.5);
\draw[convexborder,dotted] (-1.5,2) -- (-2.25,2) (2,-1.5) -- (2,-2.25);

\draw[shsborder,green!60!black] (-1.5,1.95) -- (-0.05,1.95) -- (-0.05,-1.5);
\draw[shsborder,red!50] (-1.5,-0.05) -- (1.95,-0.05) -- (1.95,-1.5);
\draw[shsborder] (-1.5,2.05) -- (2.05,2.05) -- (2.05,-1.5);
\end{scope}

\end{tikzpicture}
\end{center}
\caption{Left: A representation of a lift of the polytope of Figure~\ref{fig:perturbation}(a) which does not satisfy the first part of Theorem~\ref{ConjDevelinYu}. This lift is represented by its section with the hyperplane $\{\liftx \in \K^3 \mid \liftx_1 = 1 \}$, and by instantiating the formal parameter $t$ by a sufficiently small positive value. Right: the intersection of the images under the valuation map of the facet-defining half-spaces of the lift consists of $\PP$ along with the (unbounded) set of points of the form $(0,1,\lambda)$ and $(0,\lambda,1)$ for $\lambda \in \R_{\leq 0}$ (the boundaries of the half-spaces are depicted by dashed lines).}\label{fig:tropicalization}
\end{figure}

In a similar way, we remark that Theorem~\ref{ConjDevelinYu} cannot be extended to non-pure tropical polytopes, even if their generators are in general position. As an example, consider the tropical polytope $\PP$ generated by the points $v^1 = (0,-1,1)$, $v^2 = (0,0,0)$ and $v^3 = (0,1,-1)$, which are in general position (this polytope has been depicted on the left-hand side of Figure~\ref{fig:non-pure} above). Thanks to the characterization given in Proposition~\ref{prop:facet_characterization}, it can be verified that the facet-defining half-spaces of any lift of $\PP$ provide the tropical half-spaces $\HH((0,0,1),\{2,3\})$, $\HH((0,1,0),\{2,3\})$ and $\HH((0,1,1),\{1\})$. These half-spaces are precisely the half-spaces whose boundaries are depicted on the right-hand side of Figure~\ref{fig:tropicalization}, and so their intersection (which is an unbounded set) strictly contains $\PP$. 
\end{remark}

\section*{Acknowledgements}

Both authors are indebted to Pascal Benchimol and St{\'e}phane Gaubert for several insightful discussions on tropical geometry.

\bibliographystyle{alpha}

\begin{thebibliography}{RGST05}

\bibitem[ABGJ15]{AllamigeonBenchimolGaubertJoswigSIDMA2015}
X.~Allamigeon, P.~Benchimol, S.~Gaubert, and M.~Joswig.
\newblock Tropicalizing the simplex algorithm.
\newblock {\em SIAM Journal on Discrete Mathematics}, 29(2):751--795, April
  2015.

\bibitem[AD09]{ArdilaDevelin}
F.~Ardila and M.~Develin.
\newblock Tropical hyperplane arrangements and oriented matroids.
\newblock {\em Mathematische Zeitschrift}, 262(4):795--816, 2009.
\newblock E-print \arxiv{0706.2920}.

\bibitem[AGG13]{AllamigeonGaubertGoubaultDCG2013}
X.~Allamigeon, S.~Gaubert, and E.~Goubault.
\newblock Computing the vertices of tropical polyhedra using directed
  hypergraphs.
\newblock {\em Discrete \& Computational Geometry}, 49(2):247--279, 2013.
\newblock E-print \arxiv{0904.3436v4}.

\bibitem[AGG14]{AGG2013}
M.~Akian, S.~Gaubert, and A.~Guterman.
\newblock Tropical cramer determinants revisited.
\newblock In G.~L. Litvinov and S.~N. Sergeev, editors, {\em Tropical and
  Idempotent Mathematics and Applications}, volume 616 of {\em Contemporary
  Mathematics}, pages 1--45. American Mathematical Society, 2014.
\newblock E-print \arxiv{1309.6298}.

\bibitem[AK13]{AllamigeonKatzJCTA2013}
X.~Allamigeon and R.~D. Katz.
\newblock Minimal external representations of tropical polyhedra.
\newblock {\em Journal of Combinatorial Theory, Series A}, 120(4):907--940,
  2013.
\newblock Eprint \arxiv{1205.6314}.

\bibitem[BSS07]{BSS}
P.~Butkovi{\v{c}}, H.~Schneider, and S.~Sergeev.
\newblock Generators, extremals and bases of max cones.
\newblock {\em Linear Algebra Appl.}, 421(2-3):394--406, 2007.
\newblock E-print \arxiv{math.RA/0604454}.

\bibitem[BY06]{blockyu06}
F.~Block and J.~Yu.
\newblock Tropical convexity via cellular resolutions.
\newblock {\em J. Algebraic Combin.}, 24(1):103--114, 2006.
\newblock E-print \arxiv{math.MG/0503279}.

\bibitem[CGQ99]{ccggq99}
G.~Cohen, S.~Gaubert, and J.~P. Quadrat.
\newblock Max-plus algebra and system theory: where we are and where to go now.
\newblock {\em Annual Reviews in Control}, 23:207--219, 1999.

\bibitem[DS04]{DS}
M.~Develin and B.~Sturmfels.
\newblock Tropical convexity.
\newblock {\em Doc. Math.}, 9:1--27 (electronic), 2004.
\newblock E-print \arxiv{math.MG/0308254}.

\bibitem[DY07]{DevelinYu}
M.~Develin and J.~Yu.
\newblock Tropical polytopes and cellular resolutions.
\newblock {\em Experimental Mathematics}, 16(3):277--292, 2007.
\newblock E-print \arxiv{math.CO/0605494}.

\bibitem[GK07]{GK}
S.~Gaubert and R.~D. Katz.
\newblock The {M}inkowski theorem for max-plus convex sets.
\newblock {\em Linear Algebra Appl.}, 421(2-3):356--369, 2007.
\newblock E-print \arxiv{math.GM/0605078}.

\bibitem[JL15]{JoswigLoho}
M.~Joswig and G.~Loho.
\newblock Weighted digraphs and tropical cones.
\newblock E-print \arxiv{1503.04707}, 2015.

\bibitem[Jos09]{joswig-2008}
M.~Joswig.
\newblock Tropical convex hull computations.
\newblock In G.~L. Litvinov and S.~N. Sergeev, editors, {\em Tropical and
  Idempotent Mathematics}, volume 495 of {\em Contemporary Mathematics}, pages
  193--212. American Mathematical Society, 2009.
\newblock E-print \arxiv{0809.4694}.

\bibitem[LMS01]{litvinov00}
G.~L. Litvinov, V.~P. Maslov, and G.~B. Shpiz.
\newblock Idempotent functional analysis: an algebraic approach.
\newblock {\em Math. Notes}, 69(5):696--729, 2001.
\newblock E-print \arxiv{math.FA/0009128}.

\bibitem[Mar02]{Marker}
D.~Marker.
\newblock {\em Model Theory: An Introduction}, volume 217 of {\em Graduate
  Texts in Mathematics}.
\newblock Springer-Verlag, New York, 2002.

\bibitem[MS15]{MaclaganSturmfels2015}
D.~Maclagan and B.~Sturmfels.
\newblock {\em Introduction to {T}ropical {G}eometry}, volume 161 of {\em
  Graduate Studies in Mathematics}.
\newblock American Mathematical Society, Providence, RI, 2015.

\bibitem[RGST05]{RGST}
J.~Richter-Gebert, B.~Sturmfels, and T.~Theobald.
\newblock First steps in tropical geometry.
\newblock In G.~L. Litvinov and V.~P. Maslov, editors, {\em Idempotent
  Mathematics and Mathematical Physics}, volume 377 of {\em Contemporary
  Mathematics}, pages 289--317. American Mathematical Society, 2005.
\newblock E-print \arxiv{math.AG/0306366}.

\bibitem[Zim77]{zimmerman77}
K.~Zimmermann.
\newblock A general separation theorem in extremal algebras.
\newblock {\em Ekonomicko-matematicky Obzor}, 13(2):179--201, 1977.

\end{thebibliography}
\def\cprime{$'$} \def\cprime{$'$}

\end{document}